\documentclass[12pt]{amsart}
\usepackage{fullpage}
\usepackage{color}
\usepackage{url}
\usepackage{graphicx} 
\usepackage{tikz}  
\usepackage[all]{xy} 
\usepackage[affil-it]{authblk} 
\usepackage{subfigure} 
\usepackage{placeins} 
\usepackage{amssymb}
\usepackage{amsmath}
\usepackage{enumerate}
\usepackage{enumitem}
\usepackage{xcolor}
\usepackage{mathtools}
\usepackage{bbding}
\usepackage{hyperref}
\usepackage{centernot}

\newtheorem{prop}{Proposition}[section]
\newtheorem*{rmk}{Remark}
\newtheorem{lem}[prop]{Lemma}
\newtheorem{ex}{Example}
\newtheorem{thm}[prop]{Theorem}
\newtheorem{cor}[prop]{Corollary}
\newtheorem{defn}{Definition}

\newcommand{\rarrow}{\rightarrow}
\newcommand{\ep}{\epsilon}
\newcommand{\om}{\omega}
\newcommand{\ZZ}{\mathbb{Z}}
\newcommand{\NN}{\mathbb{N}}
\newcommand{\RR}{\mathbb{R}}
\newcommand{\QQ}{\mathbb{Q}}

\newcommand{\ZZp}{\mathbb{Z}_{\ge0}}

\newcommand{\Inv}{\text{Inv }}

\newcommand{\id}{\text{id }}
\newcommand{\vp}{\varphi}
\newcommand{\ol}{\overline}

\newcommand{\cl}{\overset{cl}}
\newcommand{\es}{\emptyset}

\newcommand{\mf}{\mathfrak}
\newcommand{\clrel}{\overline{\mathcal{R}}}
\newcommand{\konf}{\mathfrak{K}}
\newcommand{\lonf}{\mathfrak{L}}
\newcommand{\fonf}{\mathfrak{F}}
\newcommand{\Int}{\text{int}}

\hypersetup{
    colorlinks=true,
    linkcolor=blue,
    filecolor=magenta,      
    urlcolor=cyan,
    pdftitle={Sharelatex Example},
    bookmarks=true,
    pdfpagemode=FullScreen,
    hyperindex=true,
}

\begin{document}

\begin{titlepage}

\newcommand{\HRule}{\rule{\linewidth}{0.5mm}} 

\center 
 


{ \huge \bfseries Attractors and Attracting Neighborhoods for Multiflows}\\[0.5cm] 
\textsc{\Large A Dissertation}\\[0.4cm]
\HRule \\[1.5cm]


\begin{minipage}{0.4\textwidth}
\begin{flushleft} \large
\emph{Author:}\\
Shannon Negaard-Paper
\end{flushleft}
\end{minipage}
~
\begin{minipage}{0.4\textwidth}
\begin{flushright} \large
\emph{Advisor:} \\
Dr. Richard McGehee 
\end{flushright}
\end{minipage}\\[4cm]

\textsc{\LARGE University of Minnesota}\\[1.5cm] 
\textsc{\Large Department of Mathematics}\\[0.5cm] 



{\large \today}\\[3cm] 


 

\vfill 

\end{titlepage}

\section*{Abstract}

We already know a great deal about dynamical systems with uniqueness in forward time. Indeed, flows, semiflows, and maps (both invertible and not) have been studied at length. A view that has proven particularly fruitful is topological: consider invariant sets (attractors, repellers, periodic orbits, etc.) as topological objects, and the connecting sets between them form gradient like flows. In the case of systems with uniqueness in forward time, an attractor in one system is related to nearby attractors in a family of other, ``close enough" systems. 

One way of seeing that connection is through the Conley decomposition (and the Conley index) \ref{bk:Conley}, \ref{bk:Mischaikow}. This approach requires focusing on isolated invariant sets - that is, invariant sets with isolating neighborhoods. If there is an invariant set $I$, which has an isolating neighborhood $N$, we say that $I$ is the invariant set associated to $N$, and $N$ is an isolating neighborhood associated to $I$. When the invariant set in question is an attractor or a repeller, then the isolating neighborhood is called an attracting neighborhood or a repelling neighborhood, respectively. A more specialized case may be called an attractor block or a repeller block.

This approach was expanded to discrete time systems which lack uniqueness in forward time, using relations, in \ref{bk:McGeheeAttractors} and \ref{bk:McGeheeWiandt}. Relations do not rely on uniqueness in forward time, but the graph of any map is a relation; thus they serve to generalize maps. Some of this is reviewed in the next few sections. In addition, I expanded on work done in \ref{bk:McGeheeAttractors} to show that in compact metric spaces, attractors for closed relations continue (see Section \ref{sec:CloseRelns}).

On the continuous time side, more work needs to be done. This paper is a step toward a more systematic approach for continuous time systems which lack uniqueness in forward time. This work applies to Filippov systems \ref{bk:Filippov} and in control theory \ref{pc:KJMeyer}. 
In the following pages, we establish a tool (multiflows) for discussing the continuous time case and develop a framework for understanding attractors (and therefore stability) in these systems. A crucial part of this work was establishing attractor / attracting neighborhood pairs, which happens in Section \ref{sec:AttAndAttBlk}.

\newpage
\tableofcontents

\newpage
\section{Introduction: Why Relations and Multiflows}\label{sec:motivation}

The primary objects studied in classical dynamical systems are flows, vector fields, and maps. Specifically, if an autonomous vector field $\dot{x}=f(x)$ ($x\in X$, $t\in\RR$) is locally Lipschitz continuous, then there exists a flow $\vp:X\times\RR\rarrow X$ which satisfies
	\[     \dot{\vp}(\xi,t)=f(\vp(\xi,t)) \hbox{ for all }\xi\in X,~t\in\RR \]
and the usual properties of a flow (to be reviewed explicitly in the Section \ref{sec:MapsAndFlows}). More recent work considers what happens in ``non-smooth" systems. That is, what occurs when $\dot{x}=f(x)$, where $f(x)$ is not necessarily locally Lipschitz continuous. These methods have been useful, but have yet to be tackled in a fully systematic fashion.\footnote{See comments by J. M. Guckenheimer, \ref{bk:Guckenheimer}.} Here, we propose a structure, which serves as a generalization of flows: multiflows. They have already proven useful, in ways we will discuss in a moment. For now, we look at where they fit in.

The dynamical systems we shall explore here have solutions in one of six categories of objects. These objects are categorized by three things:
\begin{enumerate}
	\item how time behaves,
	\begin{itemize}
		\item discrete
		\item continuous
	\end{itemize}
	\item existence and/or uniqueness of path in forward time,
	\item existence and/or uniqueness of path in backward time.
\end{enumerate}

The organization is as follows:

\begin{center}{\bf Organization of Dynamical Systems}\end{center}
\[\begin{array}{ c || c | c}
  & \hbox{{\bf discrete time}} & \hbox{{\bf continuous time}} \\
  \hline\hline
  \hbox{{\bf forward and }} &&\\ \hbox{{\bf backward uniqueness}} & \hbox{invertible maps} & \hbox{flows} \\
  \hline
  \hbox{{\bf forward uniqueness}} & \hbox{maps} & \hbox{semiflows} \\
  \hline
  \hbox{{\bf neither}}  &  \hbox{relations} & \hbox{``multiflows"} \end{array}\label{table:maps} \]
  
  It should be noted that, just like flows are semiflows and invertible maps are maps, so too will we regard relations and multiflows as generalizations of maps and semiflows, respectively.
  
 We are most familiar with functions. In particular, we are familiar with {\it flows} (continuous time) and {\it maps} (discrete time). In dynamical systems, invertible maps were the first discrete objects studied, as they are more closely aligned with flows, but applications provided motivation to study all maps. In these cases, we know precisely how to move forward in time from any state space. With flows, we move along a set path, which passes through some point $x$ in space, and we move exactly $t\in\RR$ units along that path, defining an orbit. For maps, we also have some point in space $x$, and we simply apply the map, possibly more than once. Iterations of the map provide the ``forward motion" for any natural number. Maps work for systems and applications where discrete time makes more sense. It is worth noting that semiflows are actually a better continuous-time analogue of maps. Only when we specify that a map $f$ is invertible are we guaranteed some way to move backward - via $f^{-1}$. If $x$ is a point in space and $f$ is an invertible map, we will be able to find $f^n(x)$ (and it will be unique) for all $n\in\ZZ$, but if we move to a slightly more general object - say a map $g$ (which may or may not be invertible), then $g^n(x)$ exists and is unique for all $n\in\ZZ_{\ge0}$. Similar statements can be made for flows and semiflows. But this introduces an asymmetry, with regard to time. Only if we generalize even further do we repair this asymmetry. 
 
We also wish to focus on the topological tools, originally developed by C. Conley and his students, to study the structure of the dynamical systems in question. McGehee, Sander, and Wiandt already did much to establish the usefulness and structure of relations as generalizations for maps \ref{bk:McGeheeAttractors},\ref{bk:McGeheeSander},\ref{bk:McGeheeWiandt}. In \ref{bk:McGeheeAttractors}, McGehee defined (among other things) attractors in systems which are determined by a relation and made many discoveries of the behavior. In \ref{bk:McGeheeSander}, McGehee and Sander proved an analog of the stable manifold theorem  for systems defined by relations, and in \ref{bk:McGeheeWiandt}, McGehee and Wiandt defined the Conley decomposition for closed relations. We now lay some groundwork for similar advancements on the continuous time side of the subject.
 
It is not difficult to find motivation to explore multiflows. For instance, the solution to $\dot{x}=|x|^{1/2}$ ($x\in X=\RR$) is in a family of the first examples one sees in a dynamical systems course ($\dot{x}=|x|^\alpha,$ $0<\alpha<1$). The function $f(x)=|x|^{1/2}$ is not locally Lipschitz continuous near $x=0$, and the solutions to $\dot{x}=|x|^{1/2}$ lack uniqueness in both forward and backward directions. The orbits are not unique, so we create a new object, which collects all the orbits: a stream. Both of those terms will be defined rigorously in Section \ref{sec:Disp}.

\begin{figure}\centering
  \includegraphics[width=2.5in]{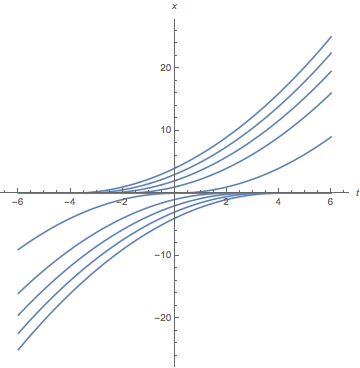}\label{fig:dispEx1}
  \caption{Some of the solutions to $\dot{x}=|x|^{1/2}$.}
\end{figure}


\begin{figure}\centering
  \includegraphics[height=2in]{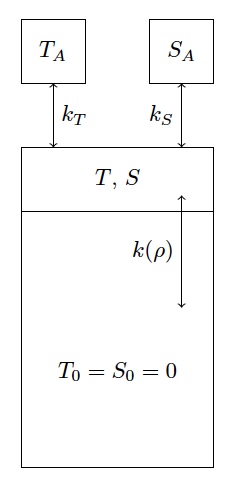}
  \caption{Setup of the Welander Box Model.}
  \label{fig:WelanderBoxModel}
\end{figure}

For already applied examples, we turn to \ref{bk:Leifeld}, in which J. Leifeld studied bifurcations that result from the Welander box model for ocean circulation, which considers the ocean as two stacked boxes: one for the shallow ocean and one for the deep ocean. The model seeks to understand the circulation between the two. Two attributes are considered - the salinity $S$ and the temperature $T$ in the shallow ocean - as these are the primary attributes in determining density $\rho$, which in turn drives ocean circulation. If the shallow ocean is less dense than the deep ocean, almost no mixing occurs. If the shallow ocean is more dense, then circulation will occur at a rate determined by the difference in density. The model is constructed, assuming that temperature $T_0$ and salinity $S_0$ in the deep ocean are constant. For simplicity, we set them to 0 (since we are only concerned with how they change). Then
	\begin{align*}
		\frac{dT}{dt}  =& k_T(T_A-T)-k(\rho)T \\
		\frac{dS}{dt} =&  k_S(S_A - S) - k(\rho)S\\
		\rho =& -\alpha T+\gamma S.
	\end{align*}
where $T_A$ and $T_S$ account for atmospheric forcing terms. Obviously, these only interact with the shallow ocean. A lot is accounted for in the convection function $k(\rho)$. 

\begin{figure}
\centering
\subfigure[Tangencies resulting at the boundary in the Welander model.]
{
	\includegraphics[height=2in]{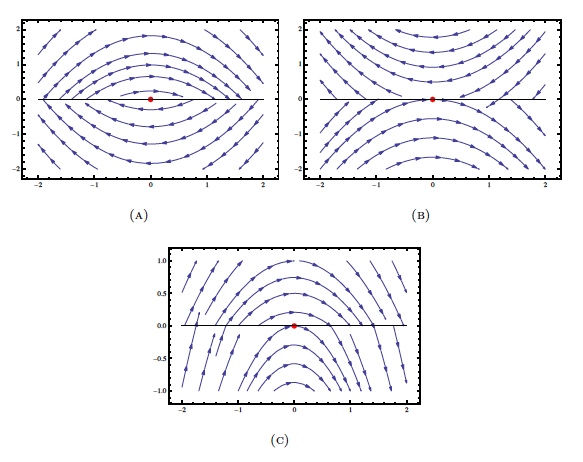}\label{fig:WelanderTangencies}
}
\subfigure[Fused focus bifurcation studied by Leifeld.]
{
	\includegraphics[height=2in]{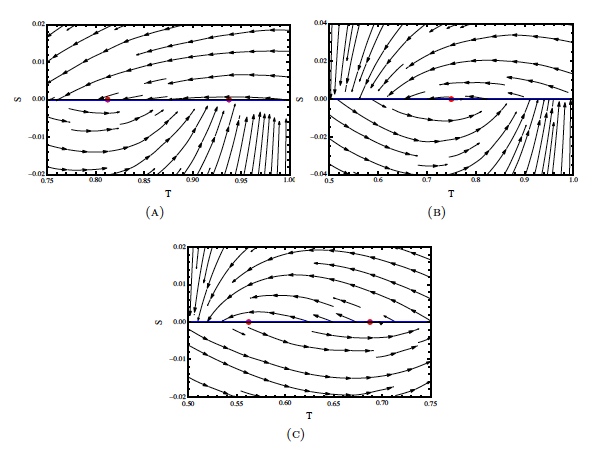}\label{fig:WelanderFusedFocus}
}
\caption{Filippov system studied by Leifeld in \ref{bk:Leifeld}}
\end{figure}

For more details, see \ref{bk:Leifeld}. We focus on this example because it is an example of a Filippov system. For a rigorous definition, see \ref{bk:Filippov}. Leifeld guides the reader through the various changes of variables and values of parameters which produce the sliding regions and fused focus bifurcation pictured in Figures \ref{fig:WelanderTangencies} and \ref{fig:WelanderFusedFocus}.

Recent work by C. Thieme proved that Filoppov systems are in fact a subset of multiflows. R. McGehee then proved that any system defined by a continuous vector field results in solutions which are multiflows. 

There is also a result by K. J. Meyer, who was studying a particular framework with applications to ecology and climate modeling. Let $x\in X$, a compact space and consider the differential equation
\begin{align}\dot{x}=&f(x)+g(x),\end{align}\label{eq:KateODE}
where $f$ is locally Lipschitz continuous. Then the system of solutions
\begin{align*}     \Phi=\{(t&,x,y)\in[0,\infty)\times X\times X : \hbox{ there is some $g$ with }\|g\|_\infty\le r , g\in L^1_{\text{loc}}\cap L^\infty \\
				&\hbox{such that (\ref{eq:KateODE}) has a solution $\chi$ satisfying }\chi(0)=x,~\chi(t)=y\}.
\end{align*}
is also a multiflow. Its motivation comes in the study basins of attraction for systems which might be perturbed or controlled in non-smooth ways (including systems in which errors in measurement might occur).


My own contributions are any theorems in sections \ref{sec:Relns}, \ref{sec:Disp}, \ref{sec:UltBehavior}, \ref{sec:CloseRelns}, and Appendix \ref{sec:AppendixA}, except those which are explicitly labeled as coming from one of my references. In addition to some theorems, which I hope will shed light on the general structure within multiflows, I proved the semicontinuity of attractors for relations, the existence of attractors for multiflows, and the existence of attracting neighborhoods arbitrarily close to those attractors for multiflows.

\newpage
\section{Review of Maps and Flows}\label{sec:MapsAndFlows}

Flows and maps are useful and fairly well understood. For a vector field $\dot{x}=F(x)$, if $F(x)$ is locally Lipschitz, then the solutions will be flows. 
We frequently use maps as tools to understand what happens in a flow after a set period of time. Morally, $f(x)=\phi(T,x)=\phi^T(x)$ defines a map on the point $x\in X$, where $T$ is usually some set time in $\RR$. They are also used in studying systems, which benefit from a discrete perspective. Examples include populations of animals or measurement of sea ice, which have distinct seasons of growth. While we are guaranteed the ability to flow forward and backward in time, in each path of our flow, the same cannot be said for all maps. That is, even if $f$ is a map, there is no guarantee it has an inverse.

With that in mind, we come to the need to refer to the non-negative halves of our time domains. We use the convention that $\ZZ_{\ge0}=\{0,1,2,\ldots\}$ and $[0,\infty)$ refers to the non-negative interval in $\RR$. 
We use the classic setting for dynamical systems - systems with forward uniqueness. Later we will focus on the analogs with respect to relations and multiflows. 

\begin{defn}
Let $\vp:\RR\times X\rightarrow X$ be a {\bf flow} if it is a function that is continuous in both variables and satisfies
\begin{enumerate}
	\item $\vp(0,x) = x$ (and $x$ is therefore called the {\bf initial point}.),
	\item $\vp(s+t,x)=\vp\left(s, \vp(t,x)\right)=\vp\left(t, \vp(s,x)\right)$,
\end{enumerate}
for all $s,t\in\RR$, $x\in X$. Also, let $\vp^t(x):=\vp(t,x)$. 
\end{defn}

\begin{defn} Let $\vp:[0,\infty)\times X\rightarrow X$ be a {\bf semiflow} if it is a function that is continuous in both variables and satisfies
\begin{enumerate}
	\item $\vp(0,x)=\vp_0(x) = x$,
	\item $\vp(s+t,x)=\vp\left(s, \vp(t,x)\right)=\vp\left(t, \vp(s,x)\right)$,
\end{enumerate}
for all $s,t\in[0,\infty)$, $x\in X$. \end{defn}

Note that a semiflow is defined for forward time, but not necessarily for negative $t$-values, or backward time. This makes a semiflow a little more similar to a (possibly noninvertible) map than a flow is. 

We might also use the notation $\vp^t: X\rightarrow X$, which is defined by the image of each $x\in X$, after following the flow through time $t\in[0,\infty)$. Thus, we may speak of flows in terms maps. In fact, we call $\vp^t$ a {\it fixed time map}.

Flows have been largely studied because of their usefulness as solutions of differential equations, but flows have also been studied in their own right. For instance, not every flow comes as the solution to a smooth vector field. 
Yet, those flows are studied, too. The objects themselves have proven to be interesting. 

We now define the discrete side of things. Maps are useful objects of study for a few reasons: they better represent discrete systems, and they help us understand some flows better. That is, we can break flows into their fixed time maps.

\begin{defn}
Let $f:X\rightarrow X$ be a function. We call $f$ a {\bf map}. Then we establish the following notation:
\begin{itemize}
	\item $f^0=\id$, the identity map, and
	\item $f^{n+1}=f\circ f^n = f^n\circ f$, for all $n\in\NN$.
\end{itemize}
\end{defn}

As a consequence, we get $f^{n+m}=f^n\circ f^m$. The notation for flows and maps is intentionally similar. Indeed, there is a similar structure. Given a flow $\vp$ or an invertible map $f$, there are associated families, which form groups:
\[                   \{\vp^t:X\rightarrow X\}_{t\in\RR} \hbox{ and }    \{f^n:X\rightarrow X\}_{n\in\ZZ}.\] 
Composition is the group action. If $\vp$ is only a semiflow or $f$ is a non-invertible map, then we only get a semigroup and are restricted to $t\in[0,\infty)$ and $n\in\ZZ_{\ge0}$, but composition is still the semigroup action.

Notice, however, that if $f$ is not invertible, then $f^n$ is not necessarily a map for $n<0$. The asymmetry causes complications. This will motivate several new definitions, when we try to generalize the notion of {\it invariance} to relations and multiflows. We do not discuss the complications much here, but additional information can be found in \ref{bk:McGeheeAttractors}.

\subsection{Orbits}

We are frequently concerned with one specific solution (or a collection of specific solutions). Starting at a particular point, what values can be obtained? In what order are they obtained? We therefore need the concept of an {\it orbit}. 

\begin{defn}
Let $\vp: \RR\times X\rarrow X$ be a flow, and let $x\in X$ be an initial point. The {\bf orbit of $\vp$ containing $x$} is the set $O_{\vp,x}=\bigcup\limits_{t\in\RR}\vp(t,x),$ with the obvious ordering (on $t$).
\end{defn}

We preserve the {\it order} in which points were obtained, but the orbits themselves are just collections of points. If there is some $y\in X$ where $\vp(t,x)=y$ for some $t$, then $O_{\vp,x}=O_{\vp,y}$. A result of uniqueness in flows is that for any two orbits $\vp(t,x)$ and $\vp(t,y)$, they are either exactly equal, or their intersection is empty.

One wonders how much of the rich structure of flows can be preserved, if one removes the necessity that a flow be a function. We shall look more closely later, when we define multiflows (Section \ref{sec:Disp}).

On the discrete side we define orbits for maps.

\begin{defn}
Let $x\in X$ be a fixed point, and let $f$ be a map (with or without inverse). Then let $x_0=x$ and define the {\bf orbit of $f$ through $x$} to be the sequence
		\[ O_{f,x}= \left\{ \begin{array}{ll} \{\ldots,x_{-1},x_0,x_1,\ldots\}, & \hbox{if }f\hbox{ is invertible, or }\\
								\{x_0,x_1,\ldots\}, & \hbox{if }f\hbox{ is not invertible}, \end{array}\right.\]
where $x_{n+1}=f(x_n)$.
\end{defn}

We bring this up to contrast it with the notions of orbits and streams for relations and multiflows.

\subsection{Fixed Points and Periodic Orbits}\label{sec:FixedPeriodic}

One of the key features in any dynamical system is the invariant set. Indeed, we use these sets to understand the overall structure of a flow or map. Between them are gradient like flows (or discrete connecting orbits in the case of maps). We will get to a formal definition of an invariant set (Section \ref{sec:UltBehaviorRelns}), but first we consider the most common examples: periodic orbits and fixed points.

\begin{defn}\label{orbitflow}
If $\vp:X\times \RR \rarrow X$ is a flow, then there is a {\bf periodic orbit} of length $t>0$ if $\vp^t(x)=x$, and $t$ is the smallest positive number for which this happens. That orbit is $\{\vp^s(x) : 0\le s<t\}.$
\end{defn}

Actually, because of the nature of a periodic orbit, and by the uniqueness property of flows, if $\vp$ has a periodic orbit of length $t$ through the point $x$, then
	\[ \bigcup\limits_{s\in\RR} \vp^s(x) = \{\vp^s(x) : 0\le s<t\}.\]
That is, all of a periodic orbit is accounted for in a single trip around the orbit. In Example \ref{ex:rotnPerOrb}, we consider one of the simplest examples that results in periodic orbits.

\begin{defn}
A {\bf fixed point} for a flow $\vp$ is a point $x\in X$ such that $\vp^t(x)=x$ for all $t\ge0$. 
\end{defn}

\begin{ex}\label{ex:rotnPerOrb}
Let $\vp: \RR\times \RR^2 \rarrow \RR^2$ be the solution to
	\[ \left[\begin{array}{c} \dot{x} \\ \dot{y} \end{array}\right] = \left[\begin{array}{cc} 0 & 1 \\ -1 & 0 \end{array}\right]\left[\begin{array}{c} x \\ y \end{array}\right].\] 
	Then
	\[\vp^t(x_0,y_0)=\left[\begin{array}{c} x \\ y \end{array}\right] = \left[\begin{array}{cc} \cos t & \sin t \\ -\sin t & \cos t \end{array}\right]\left[\begin{array}{c} x_0 \\ y_0 \end{array}\right].\] 
	A few orbits are displayed in Figure \ref{fig:rotnPerOrb}. If $(x_0,y_0)=(0,0)$, we have a fixed point (equilibrium). If $(x_0,y_0)\neq(0,0)$, then the orbit containing this point is a circle, centered at the origin, and its length is $2\pi$.
	
\begin{figure}\centering
  \includegraphics[width=.4\linewidth]{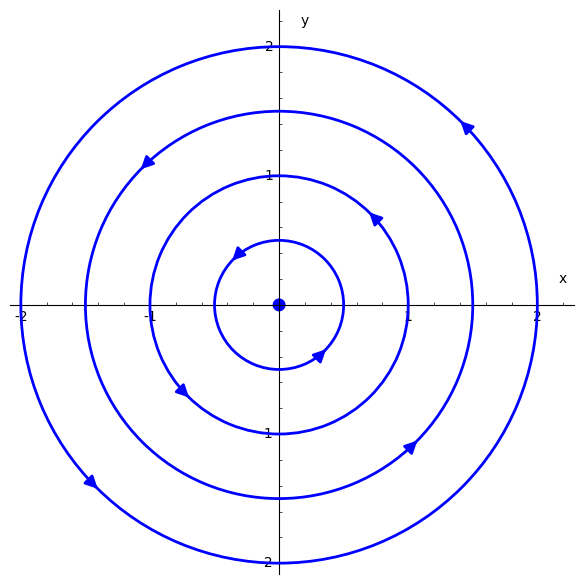}
  \caption{Orbits of the solution in Example \ref{ex:rotnPerOrb}.}
  \label{fig:rotnPerOrb}
\end{figure}
\end{ex}

\begin{defn}\label{orbitmap}
If $f:X\rightarrow X$ is a map, then there is a {\bf periodic orbit} of length $n\in \NN$ if $f^n(x)=x,$ and $n$ is the smallest positive integer for which this happens. The orbit itself is $\{f^m(x) : 0\le m<n\}.$
\end{defn}

\begin{defn}
A {\bf fixed point} for a map $f$ is a point $x\in X$ such that $f(x)=x$. 
\end{defn}

In other words, a fixed point for a map $f$ has an orbit of length 1. Whereas, $x\in X$ is a fixed point for the flow $\vp$ if there is no smallest $t>0$ for which $\vp^t(x)=x$. In this case, $\vp^t(x)=x$ for all $t\in\RR$. That is, fixed points can be viewed as a special case of a periodic orbit. Fixed points are also called {\it rest points}, {\it equilibrium points}, or simply {\it equilibria}.

\begin{ex}\label{ex:mapPerOrbit}

Let $f$ act on the set of points $X=\{a,b,c,d,e\}$ by
	\begin{align*}
		f(a)&=b\\
		f(b)&=c\\
		f(c)&=d\\
		f(d)&=e\\
		f(e)&=c.
	\end{align*}
This is more easily seen in the illustration Figure \ref{fig:MapEx}.

	\begin{figure}[h]\centering
		\includegraphics[width=1.7in]{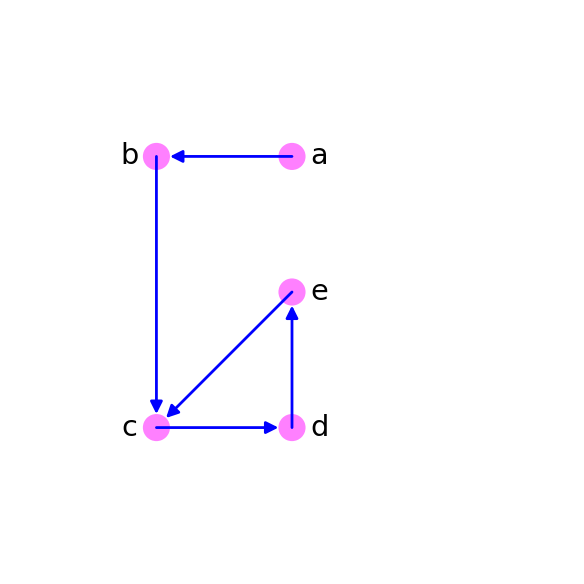}
  		\caption{Illustration of the map $f$ from Example \ref{ex:mapPerOrbit}.}\label{fig:MapEx}
	\end{figure}

There is a periodic orbit of length 3:
	\[f(c)=d,~ f^2(c)=e,~f^3(c)=c.\]
Recall that any point in a period orbit can be used as the starting point, and the period will have the same length. Thus, the orbit through $d$ is the same orbit and has the same length: $f^n(d)=d \iff 3|n$ (and $n\ge0$ because $f$ is not invertible).

Because $a$ and $b$ {\it eventually} (after enough iterations of $f$) map to one of $c,~d,$ or $e$, nothing in the system escapes this orbit. We will revisit this example when we discuss invariant sets in more depth. 
\end{ex}

\newpage
\section{Relations}\label{sec:Relns}

We begin by comparing relations to maps. A {\em map} $f:X\rarrow X$ sends each $x\in X$ to a unique element $y\in X$, or $f:x\mapsto y$. Another way of thinking is that $f$ sends each $x\in X$ to a one-element subset $\{y\}$ of $X$. The graph of any map is a relation, and indeed the graph of any continuous map is a closed relation. If $f$ is invertible then $f^{-1}$ is also a map; the graph of $f^{-1}$ is also a relation.

We expand this notion to incorporate map-like objects in which the image of $x$ may be any subset of $X$.

\begin{defn} A {\it relation} $f$ on a space $X$ is a subset $f\subset X\times X$.
\end{defn}

We typically deal with closed relations $f\overset{cl}{\subset} X\times X$, as graphs of continuous maps are closed relations.
%
A map $f$ is just a relation such that for every $x\in X$, there exists a unique $y\in X$ such that $(x,y)\in X$. We usually regard $y$ as the image of $x$ under $f$. For dynamical systems which lack uniqueness, we simply take away that restriction. Thus, relations are a natural generalization of maps. Many of the properties associated with maps carry over. Many do not. 

\begin{defn} Let $f$ be a relation\footnote{We use $f$, despite the fact that these rarely correspond to functions. This is for historic purposes and to encourage a perspective, which treats relations as set maps.} on $X$, and let $\subset X$. The {\it image} of S under $f$ is the set
\[ f(S)\equiv \{y\in X: \hbox{ there is some }x\in S\hbox{ satisfying }(x,y)\in f\}.\]
\end{defn}

As with all modern mathematics, we stand on the shoulders of giants. This section owes its existence to \ref{bk:McGeheeAttractors}. My knowledge of relations (in the context of dynamical systems) comes almost exclusively from that paper and its author. I use the definitions and theorems from that paper because they are the foundation for my work with multiflows. I am skimming, using only the tools I need. Anyone wishing to study topological features of discrete time dynamical systems with non-uniqueness in forward time should turn first to that paper.

We will use the conventions: for a relation $f$ on $X$ and $x,y\in X$,
 \begin{itemize}
 \item $f(x)\equiv f\left(\{x\}\right),$ and
 \item $y\in f(x) \iff (x,y)\in f.$
 \end{itemize}

The following general results in relations will prove invaluable. See \ref{bk:McGeheeAttractors} for their proofs.

\begin{lem}[Lemma 1.1 in \ref{bk:McGeheeAttractors}]\label{lem:relninclusion} If $f$ and $g$ are relations on $X$ with $S,T\subset X$, then
\begin{enumerate}
\item $S\subset T\implies f(S)\subset f(T)$ and
\item $f\subset g\implies f(S)\subset g(S)$.
\end{enumerate}
\end{lem}

\begin{lem}[Lemma 1.4 in \ref{bk:McGeheeAttractors}]\label{lem:relnUandInt} If $f$ is a relation on $S$ and if $\mf{S}$ is a set of subsets of $X$, then
\begin{enumerate}
\item $f\left(\bigcup\limits_{S\in\mf{S}}S\right)         =\bigcup\limits_{S\in\mf{S}}f(S)$
\item $f\left(\bigcap\limits_{S\in\mf{S}}S\right)\subset\bigcap\limits_{S\in\mf{S}}f(S)$
\end{enumerate}
That is, $f$ preserves unions, but not intersections.
\end{lem}

\begin{ex}\label{ex:RelnsDoNotPresIntn}
Let $X=\{1,2,3\}$ and let $A=\{1,2\}$, $A'=\{2,3\}$, and let $f$ be the relation on $X$ such that
	\begin{align*}
		f(1)=&\{1,2\}\\
		f(2)=&\es\\
		f(3)=&\{2,3\}.
	\end{align*}
Then, $A\cap A'=\{2\}$, and $f(A\cap A')=\es\subsetneq \left(f(A)\cap f(A')\right) = \{2\}.$
\end{ex}

There are two properties from basic set theory that will also be useful. We combine them into one lemma. In \ref{bk:McGeheeAttractors}, this was Lemma 2.6. We use the shorthand notation: if $\mathfrak{S}$ is a set of subsets $\{S\}$ of a set $X$, then \[\bigcap \mathfrak{S}\equiv\bigcap\limits_{S\in\mathfrak{S}} S.\]

\begin{lem}{Theorem 2.6 in \ref{bk:McGeheeAttractors}}\label{lem:intersections}
If $\mathfrak{S}$ and $\mathfrak{T}$ are sets of subsets of a set $X$, then the following hold:
\begin{enumerate}
	\item If $\mathfrak{S}\subset\mathfrak{T}$, then $\bigcap\mathfrak{T}\subset\bigcap\mathfrak{S}$, and
	\item If for every $T\in\mf{T}$ there is some $S\in\mf{S}$ such that $S\subset T$, then $\bigcap\mf{S}\subset\bigcap\mf{T}.$
\end{enumerate}
\end{lem}

Before moving on, we establish one more lemma, to do with the closedness of images under relations.

\begin{lem}\label{thm:closedImages}
If $f$ is a closed relation over a compact space $X$, then $f$ acts as a closed multivalued map. That is, if $S\cl{\subset}X$, then $f(S)\cl{\subset}X$.
\end{lem}

\begin{proof}
We begin by thinking of $f\subset X\times X$ in its graphical form. Then, we may speak of projections, say $\pi_1$ to project to the first copy of $X$ and $\pi_2$ to project to the second, as described in detail in \ref{bk:McGeheeAttractors}. So, considering the images of $f$ is to use $\pi_2(f)$. Projections on compact sets are closed maps.

To see this in further detail, let $S$ be any closed subset of $X$. Then, $S\times f(S)=f\cap(S\times X)$, which is closed, and $f(S)=\pi_2(f\cap(S\times X))$. Again, because projections on compact sets are closed, we are guaranteed that $f(S)$ is closed.
\end{proof}

\subsection{Composition of Relations}

We go forward in discrete time by composing relations the same way one composes maps. That is,
	\begin{align*}
		f^2(x)=&f(f(x)),\\
		f^3(x)=&f(f^2(x)), \hbox{ etc.}
	\end{align*}
\begin{defn}If $f$ is a relation over $X$, then for $n\in\ZZ_{\ge0}$ $f^n$ is also a relation over $X$ defined by
	\[f^0=\id_X=\{(x,x):x\in X\},\]
	\[f^n=\{(x,z): z=f(f^{n-1}(x))\} \hbox{ if }n>0.\]
Therefore, $f^n(x)=\{z: (x,z)\in f^n\}$ and if $S\subset X$, \[f^n(S)=\{z: \hbox{ there is some }x\in S\hbox{ such that }(x,z)\in f^n\}.\]
\end{defn}

While relations do not share all of the properties of maps, there is quite a lot of structure that is preserved in moving to this generalization. In fact, there is a semigroup structure on the compositions of $f$.

\begin{lem}
Let $f$ be a relation over $X$. Then $\{f^n\}_{n\in Z_{\ge0}}^+$ forms a {\bf commutative semigroup}. It is an easy exercise to prove the following properties for $m,n,l\in\ZZ_{\ge0}$:
	\begin{itemize}
		\item $f^0=\id_X$ (by definition)
		\item $f^{m+n}=f^m\circ f^n = f^n\circ f^m$
		\item $f^{l+(m+n)}=f^{(l+m)+n}$
	\end{itemize}
\end{lem}

The relation inherits this behavior directly from the semigroup $\ZZ_{\ge0}$.
We may also fix $x$ and define an orbit, which is definitely different for a relation than for a map. We will need to deal with two notions: an orbit and a stream.

\begin{defn}\label{defn:StreamReln}
Let $x\in X$ be a fixed point, and let $f$ be a relation. Then define the {\bf forward stream of $f$ at $x$} (usually just called the stream of $f$ at $x$) to be the sequence of subsets
		$$O_{f,x}=\{x,f(x),f^2(x),\ldots\}.$$
\end{defn}

\begin{rmk} Note that we are including {\bf all} of the possible paths of $x$ in these streams. We may want to follow a specific path, or an orbit. Before exploring that definition it is worth noting that if $x_i\in f^i(x)$, then $f(x_i)\subset f(f^i(x))=f^{i+1}(x)$ by Lemma \ref{lem:relninclusion} (1).
\end{rmk}

\begin{defn}
Let $x\in X$ and $f$ a relation on $X$. Then $O_{f,x}$ is defined as in Definition \ref{defn:StreamReln}. This stream potentially contains several choices of paths, or orbits. A {\bf forward orbit} (usually just called an orbit) is defined as a sequence:
	\[ \{x_0=x,x_1,x_2,\ldots\},\] 
where $x_{i+1}\in f(x_i)$ for all $i\ge0$.
\end{defn}

In the case where $f$ is a map, there is only one forward orbit, so a stream and an orbit naturally condense into one object.
We wait until Section \ref{sec:RelnTransposes} to define {\it backward orbit} and {\it backward stream} because it requires knowledge of {\it $f$-transpose}.

\begin{defn}\label{defn:relnsub-orbit}
Let $x\in X$ and $f$ a relation on $X$. Let $O_{f,x}$ be the stream at $x$. Let $y\in f^i(x)$ for some $i\ge0$. Then $O_{f,y}$, the stream at $y$, is a {\bf sub-stream} of $O_{f,x}$.
\end{defn}

A sub-orbit practically defines itself, given what we know of sub-orbits for maps and of sub-streams. 

\begin{defn}Let $\{x_0,x_1,x_2,\ldots\}$ be an orbit under $f$ at $x=x_0$, and let $y=x_i$ for some $i\ge0$. Then the {\bf sub-orbit} at $y$ is $\{y=x_i,x_{i+1},x_{i+2},\ldots\}$.
\end{defn}

\begin{ex}\label{ex:RelnOrbitPath}
Let $X=\{1,2,3\}$ and set $f=\{(1,1),(1,2),(2,1),(2,3),(3,3)\}$. Then the stream at 1 is
	\[O_{1,f}=\left\{ 1, \{1,2\}, \{1,2,3\},\{1,2,3\},\ldots\right\}.\]
Within this stream, the following are orbits:
	\[ \{1,2,1,1,\ldots\}, \{1,2,1,2,\ldots\},\{1,2,3,3,\ldots\}.\]
The following are not orbits:
	\[ \{1,2,2,\ldots\}, \hbox{ or } \{2,1,3,\ldots\}\]
because, while $2, 3\in f^2(1)$, $2\notin f(2)$ and $3\notin f(1)$. Those choices are not compatible.

\begin{figure}\centering[h]
		\includegraphics[width=1.7in]{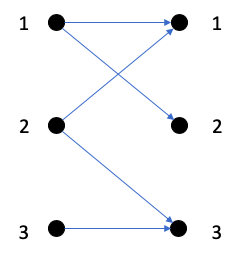}
  		\caption{Illustration of the relation $f$ from Example \ref{ex:RelnOrbitPath}.}\label{fig:RelnStreamEx}
	\end{figure}

\end{ex}

By defining orbits in an iterative way, we guarantee that each element of the path sequence exists in a compatible sub-stream. 

\subsection{Transposes}\label{sec:RelnTransposes}

For an invertible map $f$, $f^{-1}$ is a well-defined map. If $f$ is a map (not necessarily invertible), $f^{-1}$ is still a relation, known as the {\it inverse image} of $f$. Recall that a map may be regarded as a special case of a relation. So we define the inverse image in the following way.

\begin{defn}
Let $f$ be a relation on $X$ and $S\subset X$. Then the {\em inverse image} of $S$ is the set
		\[f^{-1}(S)\equiv\{x\in X:f(x)\subset S\}.\]
\end{defn}

If $f$ is a map, this is compatible with the usual definition. However, if $f$ is a relation, then $f^{-1}$ is not necessarily a relation, as Example \ref{ex:InvNotReln} demonstrates. Thus, it is not the natural object to use for dealing with ``backward time" in relations. Instead we will use $f^*$.

\begin{defn}
Let $f$ be a relation on $X$. Then the {\bf transpose of $f$} (or $f$-transpose) is
	\[ f^* = \{ (y,x) : ~ (x,y)\in f\}.\]
\end{defn}

The same conventions carry over to the transpose. For instance,
	\begin{align*}
		f^*(x)&=\{y:~ (x,y)\in f^*\}\\
			&=\{y:~(y,x)\in f\}, \hbox{ and }\\
		f^*(S)&=\{y:~ (s,y)\in f^*\hbox{ for some }s\in S\}\\
			&=\bigcup\limits_{s\in S} f^*(s).
	\end{align*}

\begin{ex}\label{ex:InvNotReln} This example comes from \ref{bk:McGeheeAttractors}. Let $X=\{0,1\}$ be 		the space, \[f=\{(0,0),(0,1),(1,0)\}\subset X \times X\] a relation, illustrated by Figure \ref{fig:InvNotReln}. Let 	$S_0=\{0\}$ and $S_1=\{1\}$ be subsets 	of $X$ and note that
	\[ f^{-1}(S_0\cup S_1) = f^{-1}(X)=X, \hbox{ but } f^{-1}(S_0)\cup f^{-1}(S_1)= \{1\}\cup\emptyset=\{1\}. \]
	Specifically, we see that $0\notin f^{-1}(S_0)$ because $f(0)=\{0,1\}\not\subset S_0$. Thus, $f^{-1}$ fails to 	preserve unions, contradicting Lemma \ref{lem:relnUandInt}. However, $f^*$ behaves in a less surprising 		way:
	\[f^*(S_0\cup S_1)=\{0,1\}=X, \hbox{ and } f^*(S_0)\cup f^*(S_1) = \{0,1\}\cup \{0\} = X.\]

	\begin{figure}\centering
	    \subfigure[as relation on $X\times X$]
	    {
	        \includegraphics[width=.35\linewidth]{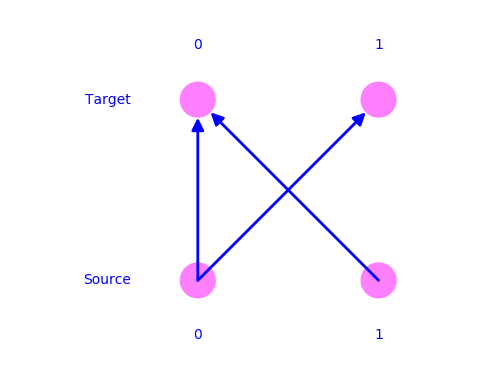}
	        \label{fig:INR2}
	    }
	    \subfigure[as set map on $X$]
	    {
	        \includegraphics[width=.35\linewidth]{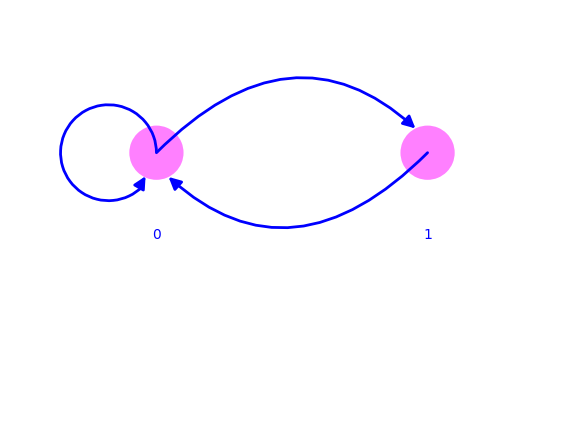}
	        \label{fig:INR1}
	    }
	    \caption{Illustrations of the relation $f$ from Example \ref{ex:InvNotReln}.}
	    \label{fig:InvNotReln}
	\end{figure}
\end{ex}

In understanding why the transpose is the more natural object for ``backward time," we return to the graphical interpretation of a relation. Recall that the graph of a map is a closed relation. If a map $F$ is invertible, and the relation $f$ gives the graph of $F$, then the graph of $F^{-1}$ will be $f^*$.
Note that $(f^*)^n=(f^n)^*$, and because $f^*$ is a relation $\{(f^*)^n\}_{n\in\ZZ_{\ge0}}^+$ forms a semigroup.

The transpose and inverse image of a relation are related in the following way.
\begin{thm}[Lemma 1.9 in \ref{bk:McGeheeAttractors}]\label{TransposeInvIm}
Let $f$ be a relation on a space $X$ and let $S\subset X$. Then
	\[ f^*(S)^c=f^{-1}(S^c).\]
\end{thm}

We can now define backward orbits and streams.

\begin{defn}\label{defn:RelnBwdOrbStr}
Let $f$ be a relation on a space $X$, define the {\bf backward stream of $f$ at $x$} to be the sequence of image sets
		$$O_{f^*,x}=\{x,f^*(x),(f^*)^2(x),\ldots\},$$
and let a {\bf backward orbit of $f$ at $x$} be a sequence of elements $\{x_0=x,x_{-1},x_{-2},\ldots\}$, where $x_{-(i+1)}\in f^*(x_{-i})$ for all $i\ge0$.
\end{defn}

Because the backward orbits and backward streams are defined as forward orbits and forward streams of another relation (the transpose of the original), we usually drop the ``forward" when discussing forward streams and forward orbits.

\newpage
\section{Multiflows}\label{sec:Disp}

Here we explore the object, which completes the chart - an object which models a dynamical system, potentially lacking in forward uniqueness (like relations), but which works in continuous time (like flows and semiflows). The answer is dispserions.

\begin{center}{\bf Organization of Dynamical Systems}\end{center}
\[\begin{array}{ c || c | c}
  & \hbox{{\bf discrete time}} & \hbox{{\bf continuous time}} \\
  \hline\hline
  \hbox{{\bf forward and }} &&\\ \hbox{{\bf backward uniqueness}} & \hbox{invertible maps} & \hbox{flows} \\
  \hline
  \hbox{{\bf forward uniqueness}} & \hbox{maps} & \hbox{semiflows} \\
  \hline
  \hbox{{\bf neither}}  &  \hbox{relations} & \hbox{multiflows} \end{array}\label{table:maps} \]

Begin by recalling the definition of a flow:
{\it Let $\vp:\RR\times X\rightarrow X$ be a flow if it is a function that is continuous in both variables and satisfies
\begin{enumerate}
	\item $\vp(0,x) = x$ (and $x$ is therefore called the {\bf initial point}.),
	\item $\vp(s+t,x)=\vp\left(s, \vp(t,x)\right)=\vp\left(t, \vp(s,x)\right)$,
\end{enumerate}
for all $s,t\in\RR$, $x\in X$, and we use the notation $\vp^t(x):=\vp(t,x)$. }

We build the continuous time analog of a relation, which must also serve as a generalization of flows and semiflows, allowing for lack of existence / uniqueness in forward time.

\begin{defn}
Let $X$ be a space. 
 We call $\Phi\cl{\subset}[0,\infty)\times X\times X$ a {\bf multiflow}, with the notation $\Phi^t=\{(x,y):~(t,x,y)\in\Phi\}$, if it satisfies
	\begin{enumerate}
		\item $\Phi^0=id,$ and
		\item $\Phi^{s+t}=\Phi^s\circ\Phi^t,$
	\end{enumerate}
for $s,~t\ge0$.
\end{defn}

While first encountering multiflows - especially if one is used to dealing with the well-behaved nature of flows - there are some seemingly pathological behaviors. For instance, is the semigroup property and $\Phi^0=id_X$ enough to guarantee a certain kind of continuity? Is it enough to guarantee we don't have ``gaps" - sections of an orbit with empty images, followed by non-empty images? We have begun to answer these questions.

\begin{thm}
Let $X$ be a space, and let $\Phi\subset [0,\infty)\times X\times X$ be a multiflow over $X$. Let $S\subset X$ (and note that $S$ may be a single point). If there is any $s\in[0,\infty)$ for which $\Phi^s(S)=\es$, then $\Phi^t(S)=\es$ for all $t\ge s$.
\end{thm}

\begin{proof}
Let $X$, $S$, $\Phi$ and $s$ be as described. Then for all $t\ge s$, $t=s+\alpha$ for some $\alpha\ge0$, which means
	\[\Phi^t(S)=\Phi^{s+\alpha}(S) = \Phi^\alpha(\Phi^s(S))=\Phi^\alpha(\es)=\es,\]
because the image of the empty set, under a relation, is the empty set.
\end{proof}


\begin{thm}{K. J. Meyer \ref{pc:KJMeyer}} Let $Z$ be a compact metric space and let $K$ be a closed subset of $\mathbb{R}\times Z$. Let $K^t=\{z\in Z: (t,z)\in K \}$. Then the map $M:\mathbb{R}\to \mathcal{P}(Z)$ given by $t\mapsto K^t$ is semicontinuous; that is, given $t\in\mathbb{R}$ and $\epsilon>0$, there exists a $\delta>0$ such that $|s-t|<\delta$ implies $K^s\subset N_\epsilon(K^t)$.\\
\end{thm}

\begin{proof} Assume the negation for the sake of contradiction. Suppose that there is a $t\in\mathbb{R}$ and an $\epsilon>0$ such that for all $\delta>0$, there is some $s$ with $|s-t|<\delta$ and also $K^s\not\subset N_\epsilon(K^t)$. We will show that this leads to the conclusion that $K$ does not contain a limit point, contradicting $K$ being closed. \\
\begin{figure}\centering
\includegraphics[scale=0.5]{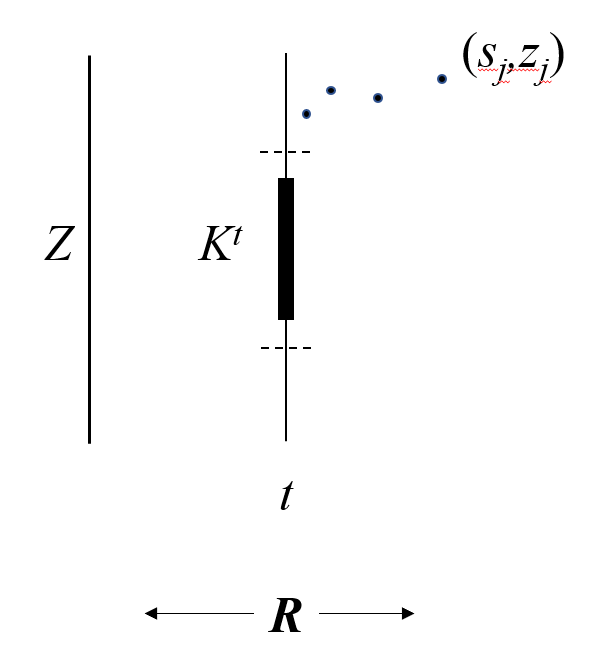}\caption{An example of a convergent subsequence.}
\end{figure}\label{fig:convSubSeq}
Given the assumption, there must be a sequence of points in $K$, $\{(s_j,z_j)\}_{j=1}^\infty,$ such that $|s_j-t|<\frac{1}{j}$ and $z_j\not\in N_\epsilon(K^t)$. Since $Z$ is compact, $\{z_j \}_{j=1}^\infty$ has a convergent subsequence $\{z_{j_n}\}_{n=1}^\infty$ converging to some $z^*\not\in K^t$. Then the sequence of points $\{(s_{j_n},z_{j_n})\}_{n=1}^\infty\subset K$ has $(t,z^*)$ as a limit point. Since $z^*\not\in K^t$, $(t,z^*)\not\in K$. So $K$ does not contain one of its limit points. But this contradicts $K$ being a closed subset of $\mathbb{R}\times Z$. \end{proof}

Note that nothing much changes if we say $Z$ is just a compact topological space. K. J. Meyer proved it for metric spaces, as it was applicable to her research, but we use the more general version, where given a neighborhood $N(K^t)$, we can find a $\delta>0$ such that $|s-t|<\delta$ implies $K^s\subset N(K^t)$. The proof is as above, mutatis mutandis. Let $Z=X\times X$ where $X$ is a compact topological space, and we have upper semicontinuity for multiflows.
Thus, in the same way that a relation can be thought of as a set map: $f:X\rarrow \mathcal{P}(X)$, a multiflow maps in the following way.

\begin{defn}
Let $X$ be a topological space, and let $\clrel(X)$ be the space of closed relations on $X$. Then $\Phi$ is a {\bf multiflow} if
	\[ \Phi: [0,\infty) \xrightarrow{cont.} \clrel(X), \]
	\[\Phi(0)=\id_X, \hbox{ and }\]
	\[\Phi(s+t)=\Phi(s)\circ \Phi(t) \hbox{ (known as the semigroup property).}\]
\end{defn}

Define $\Phi^t$ to be $\Phi(t)$. These will be our {\bf fixed time relations}. Note that in the case of a flow $\phi: \RR\rarrow X\times X$, $\phi^t$ was a fixed time map, or function. We are just expanding to a generalized set function.
The graph of this set map is the same as the multiflow, regarded as a subset of $[0,\infty)\times X\times X$. If $\vp$ is a flow then its graph, restricted to $[0,\infty)\times X\times X$, will be a multiflow. The same flexibility that proved useful in relations - being able to think in terms of a graph or in terms of set function - will be invaluable with multiflows.

To move backward, we once again use the {\it transpose} $\Phi^*$. 
This is because if we let $\Phi$ be a multiflow, then $\Phi^t$ must a relation for any fixed $t>0$. The same problems exist in attempting to move backward; namely, if we define $\Phi^{-t}$ using the inverse image $(\Phi^t)^{-1}$, then it might not be a relation. Thus $(\Phi^t)^{-1}$ might fail to be a multiflow when expanded to $t\in[0,\infty)$. We propose the following definition.

\begin{defn}
Let $\Phi\overset{cl}{\subset}[0,\infty)\times X\times X$ be a multiflow. Define
	\[\Phi^* = \{ (t,y,x): (t,x,y)\in \Phi\}\]
to be the {\bf transpose} of $\Phi$.
\end{defn}

Then, $\{(y,x): (t,y,x)\in \Phi\}=(\Phi^*)^t = (\Phi^t)^* = \{ (y,x) : (x,y)\in \Phi^t\}.$

There may be questions of why multiflows are the right objects to serve as the generalizations of flows. We argue that results like those of McGehee and Thieme lend weight to the notion that they are indeed the natural choice. Also, McGehee's '92 paper has stood for 26 years, with relations proposed as the natural generalization for maps. Multiflows are built directly from that. Fixing a time in a multiflow yields a relation. Multiflows also respect the key structures of flows and semiflows; that is, they have the semigroup property.

We now discuss a few examples of multiflows and complications that arise by allowing for the lack of uniqueness in forward time.

\begin{ex}\label{ex:DispSlidingReg}
Let $x,y,t\in \RR$. Consider
\[\left[   \begin{array}{c}  \dot{x} \\ \dot{y}\end{array}  \right] = 
\left\{    \begin{array}{cl}    
\left[ \begin{array}{c} 
1 \\ 1
 \end{array} \right], & y\ge 0  \\    &    \\
 \left[\begin{array}{c} 1 \\ 0 \end{array}\right], &y=0 \\ & \\
\left[ \begin{array}{c} 
1 \\ -1
 \end{array} \right], & y\le 0  
\end{array} \right  .  . \] 

Then the collection of all possible solutions $\Phi$ is an example of a multiflow:

\begin{align*}
	\Phi =
	       &\left\{ \left( t,  \left[\begin{array}{c} x_0 \\ 0 \end{array}\right],  \left[\begin{array}{c} x_0+t \\ y \end{array}\right]\right):  0\le t, 0 \le y \le x_0+t  \right\} \\
	       & \cup \left\{ \left( t,  \left[\begin{array}{c} x_0 \\ 0 \end{array}\right],  \left[\begin{array}{c} x_0+t \\ y \end{array}\right]\right): 0\le t,  -(x_0+t)\le y \le 0   \right\}\\
	       &\cup \left\{ \left( t,  \left[\begin{array}{c} x_0 \\ y_0 \end{array}\right],  \left[\begin{array}{c} x_0+t \\ y_0+\text{sgn}(y_0)t \end{array}\right]\right): 0\le t \le |y_0| \right\}.
\end{align*}

Let $\left[\begin{array}{c} x_0\\y_0\end{array}\right]=\left[\begin{array}{c} -2 \\ 0 \end{array}\right]$. Then $\Phi^0(x_0)=x_0$, and indeed we could follow the multiflow $\Phi$ backward in time to gain unique preimages of $x_0$. For instance, consider all the solutions in $\Phi$, which go through $\left[\begin{array}{c}x_0\\ y_0\end{array}\right]=\left[\begin{array}{c} -2 \\ 0 \end{array}\right]$. Those solutions will all agree if one moves in the negative time direction (there is a unique backwards orbit). However, following forward in time, we get all those solutions shown in Figure \ref{fig:Filippov1}, and more.

\begin{figure}\centering
  \includegraphics[height=2in]{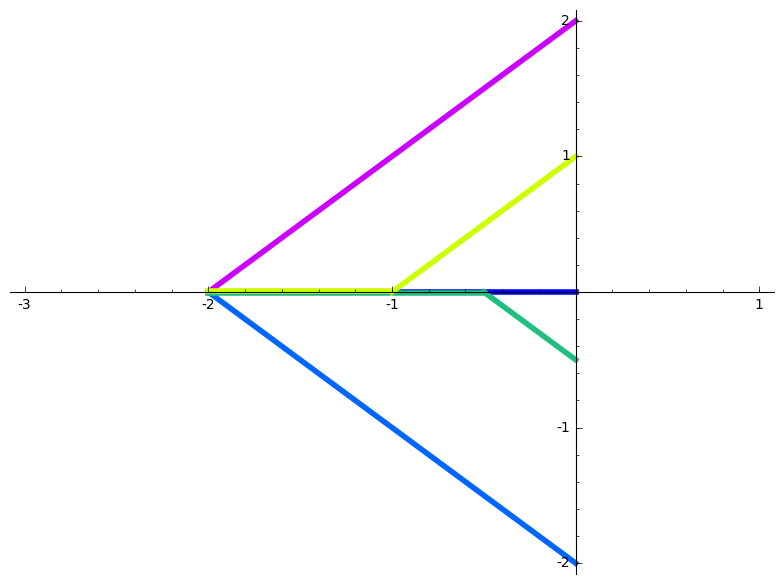}
  \caption{Some of the solutions in $\Phi$ (Example \ref{ex:DispSlidingReg}), with initial point $x_0=[-2,0]^\text{T}$, for $t\in[0,2]$.}
  \label{fig:Filippov1}
\end{figure}

In fact, $\Phi^2(x_0)=[-2,2]$. This example has a region in which sliding\footnote{For further exploration of sliding and sliding regions, see \ref{bk:Filippov} and \ref{bk:Leifeld}. The important notion here is that the vector fields are not smooth at the $x$-axis and this has the potential to cause non-uniqueness.} is possible. Coming from $x_0$, we might follow the upper or lower vector fields, or we might move along the $x$-axis. Once we leave the $x$-axis, there is only one path to follow, but at any point along the $x$-axis, we have all three choices: to move upward, to move downward, or to continue horizontally. In this case the $x$-axis is the Filippov sliding region.


\end{ex}

\begin{defn}\label{defn:StreamDispn}
Given a multiflow $\Phi\overset{cl}{\subset}[0,\infty)\times X\times X$, a {\bf forward stream} at $x\in X$ is a set of subsets of $X$, $\{U_t\}_{t\in\RR_{\ge0}}$ such that $y_t\in U_t\iff (t,x,y_t)\in \Phi$. \end{defn}

\begin{defn}\label{defn:OrbitDispn}
Given a multiflow $\Phi\overset{cl}{\subset}[0,\infty)\times X\times X$, a {\bf forward orbit} at $x\in X$ is a set of points $\{y_t\}_{t\in\RR_{\ge0}}\subset X$ such that
	\begin{enumerate}
		\item for some $t\in\RR_{\ge0}, ~ (t,x,y_t)\in \Phi,$
		\item for each $t$ there is exactly one $y_t$ in the orbit, and
		\item for any two $y_{t_0},y_{t_1}$, with $t_1>t_0$, $(y_{t_0},y_{t_1})\in\Phi^{t_1-t_0}.$
	\end{enumerate}
\end{defn}

We see that (1) is guaranteeing that $y_t$ is in the forward stream (or just stream, unless there is need to distinguish from the backward stream) and (3) forces the compatibility, like we discussed when defining streams and orbits for relations.

\begin{defn}\label{defn:BkwdStreamDispn}
Given a multiflow $\Phi\overset{cl}{\subset}[0,\infty)\times X\times X$, a {\bf backward stream} at $x\in X$ is a set of subsets of $X$, $\{V_t\}_{t\in\RR_{\ge0}}$ such that $y_t\in V_t\iff (t,x,y_t)\in \Phi^*$ (or, equivalently, $(t,y_t,x)\in \Phi$). \end{defn}

\begin{defn}
If $\Phi\overset{cl}{\subset}[0,\infty)\times X\times X$ is a multiflow and $y\in X$, then $\{x_t\}_{t\in [0,\infty)}\subset X$ is a {\bf backward orbit} if
	\begin{itemize}
	\item for some $t\in\RR_{[0,\infty)},$ $(t,y,x_t)\in\Phi^*,$
	\item for each $t$ there is at most one $y_t$ in the orbit, and
	\item for any two $x_s,~x_t$, with $t>s$, $(x_s,x_t)\in \Phi^{t-s}$.
	\end{itemize}
\end{defn}

\begin{ex}\label{ex:XOneHalf}
Let $x\in \RR^1=X$, and let $\dot{x}=|x|^{1/2}.$ The solution to this differential equation is
\begin{align*} \Phi\subset& \RR\times X\times X, \hbox{ where } \\
			\Phi=\{t,0,0\}&\cup\left\{(t,x,y):~ x\ge0,~t\ge -2\sqrt{x}, ~y=\left(\frac{t}{2}+\sqrt{x}\right)^2\right\} \\  &\cup   \left\{  (t,x,y):~      x\le0,~   t\le 2\sqrt{-x},~ y=-\left(\frac{-t}{2}+\sqrt{-x}\right)^2 \right\}.
\end{align*}

The map $f(x) = \Phi^0(x)=id.$ For each fixed $t>0$, we get a map that is ``above" the identity line $y=x$. For $t<0$, we get a map that is ``below" the identity line $y=x$. This makes sense, given that $\dot{x}=|x|^{1/2}\ge0$. Ordinarily, it makes no sense to describe $\Phi^t$ explicitly when $t<0$, but in this case it is possible to do so, and the definition agrees: for $t<0$, $\Phi^t = (\Phi^{|t|})^*$. The proof is omitted.
		
This is not a flow, as can be seen in Figure \ref{fig:XOneHalfTimeTMaps}. If one is following along the orbit $\Phi(t,0)=0$ for all $t$, we find a sliding region. Specifically, if $t\ge0$, and if $x=0$, then $t\ge -2\sqrt{x_0}$ for some $x_0$, and the first and second conditions above are satisfied. So there are two valid solutions, associated with those two conditions. Either of the first two expressions will return $\frac{\partial}{\partial t}\big|_{x=0,t=0} \Phi(t,x_0) = 0$. Therefore, $\Phi$ lacks unique solutions going forward whenever $x=0$. Likewise, $\Phi$ lacks uniqueness in the negative time direction. If $t\le0$ and $x=0$, moving backwards means continuing along $x=0$ or along the path described in the third expression. So long as $x=0$ has been chosen for each $s\in(t,0]$, the choice will again need to be made at $t$.

\begin{figure}\centering
  \includegraphics[width=2.5in]{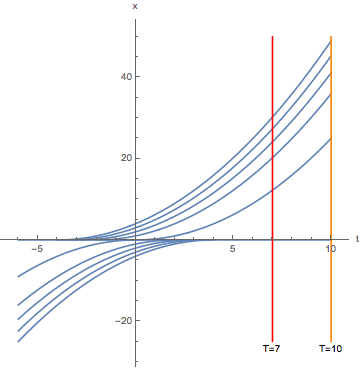}
  \caption{Solutions to $\dot{x}=|x|^{1/2}$ and $T$ markers for the fixed-time relations in Example \ref{ex:XOneHalf}. }
  \label{fig:XOneHalfTimeTMaps}
\end{figure}

\end{ex}

Since our goal was to define multiflows in such a way that when we fix a time, we get a relation, it is worth pausing to calculate some fixed-time relations. 

\begin{ex}\label{ex:XOneHalfReln}
Let us also look at the discretization of $\Phi$ in Example \ref{ex:XOneHalf}; in other words, let us see the relations that result from fixing $t$. Note that it is impossible to do so without separating the positive, negative, and zero values of $t$. The image in Figure \ref{fig:XOneHalfTimeTMaps} includes two vertical lines at fixed values $T=7$ and $T=10$. We shall eventually see the relations, which correspond to these time values.

Let $f=\Phi^0$, and obviously $f=\id$. This is not particularly interesting, but it does give us a frame with which to work. For any values of $t>0$, $\Phi^t(x)\ge x$; for any values of $t<0$, $\Phi^t(x)\le x$. That much is obvious from the ODE, since $\dot{x}\ge0$.

Let $y=f(x) = \Phi^t(x)$ for some fixed $t>0$. We must keep in mind that $x$ is the {\it starting point} and $y$ is in the image after time $t$. For $x\le\frac{-t^2}{4}$, the path is uniquely determined, and $y=-(-\frac{t}{2}+\sqrt{-x})^2$. Notice that $x<0$, and $y\le0$. Likewise, if $x>0$, there is a unique finishing value, and it is $y=(\frac{t}{2}+\sqrt{x})^2$. In this case, $x$ was already positive, and since the resulting $y$ will be more positive, no issue arises from crossing 0.

Now, we focus on $x$-values in that in-between region: $\frac{-t^2}{4} < x\le 0$. There is a time $t_x$, where $\Phi^{t_x}(x)=0$ (and $t_x$ is the smallest $t$-value for which this happens), but $t_x<t$. In fact, $t_x=2\sqrt{-x}$. We need to determine the possible paths of motion for that remaining time. On the low end, it is possible that we remain at 0 for the remainder of the time $(t-t_x)$. In this case, $0\in\Phi^t(x)$. On the high end, we have a scenario in which no measurable time is spent at 0. This means we immediately follow the positive path, starting from 0, for all remaining time $t-t_x$, and end at
	\[y = \left(\frac{t-t_x}{2} + \sqrt{0}\right)^2 = \left(\frac{ t - 2\sqrt{-x}}{2}\right)^2 = \left(\frac{t}{2} - \sqrt{-x}\right)^2. \]
All values between $0$ and $\left(\frac{t}{2} - \sqrt{-x}\right)^2$ are attainable by spending some of the remaining time in the sliding region, before following the positive path. Thus, if $x\in\left[-\frac{t^2}{4},0\right]$, then $\Phi^t(x)=\left[ 0 , \left(\frac{t}{2} - \sqrt{-x}\right)^2\right].$ Finding the image of $\Phi$, for negative values of $t$, uses similar computations.

All together, we get
\begin{itemize}
\item $t>0$:
\[
	\Phi^t(x)=\left\{ \begin{array}{ll} \left\{-\left(-\frac{t}{2}+\sqrt{|x|}\right)^2\right\}, & \hbox{if }x<-\frac{t^2}{4} \\ \\
							\left[ 0 , \left(\frac{t}{2}-\sqrt{|x|}\right)^2\right], & \hbox{if }-\frac{t^2}{4}\le x \le 0 \\ \\
							\left\{\left(\frac{t}{2} + \sqrt{|x|}\right)^2\right\} , & \hbox{if } x>0,\\
	\end{array}\right.
\]
\item $t=0$:
	\[ \Phi^t(x)=\{x\}, \]
\item $t<0$:
\[
	\Phi^t(x)=(\Phi^*)_{-t}=\left\{ \begin{array}{ll} \left\{-\left(-\frac{t}{2}+\sqrt{|x|}\right)^2\right\}, & \hbox{if }x<0 \\ \\
							\left[ - \left(\frac{t}{2}+\sqrt{|x|}\right)^2, 0 \right], & \hbox{if }0\le x \le \frac{t^2}{4} \\ \\
							\left\{\left(\frac{t}{2} + \sqrt{|x|}\right)^2\right\} , & \hbox{if } x>\frac{t^2}{4} \\
	\end{array}\right.
\]
\end{itemize}

\begin{figure}
    \centering
    \subfigure[$T=7$.]
    {
        \includegraphics[width=.35\linewidth]{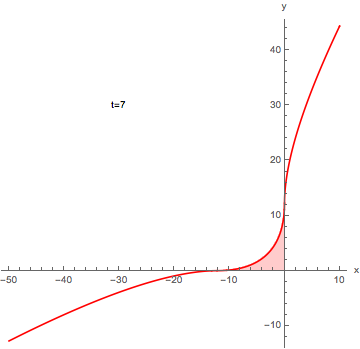}
        \label{fig:T7}
    }
    \subfigure[$T=10$.]
    {
        \includegraphics[width=.35\linewidth]{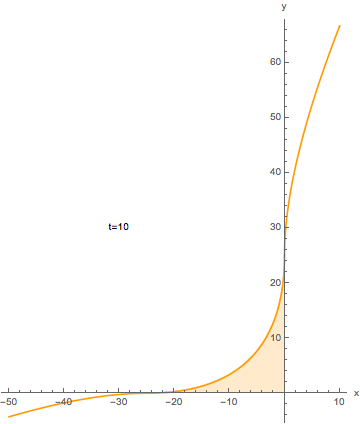}
        \label{fig:T10}
    }
    \caption{Illustrations of the fixed-time relations from Example \ref{ex:XOneHalf}.}
\end{figure}

\end{ex}

\begin{rmk}
Note that $\Phi^t$ enjoys two semigroup structures: $G_1=\{\Phi^t\}_{t\ge0} $ and $G_2=\{\Phi^t\}_{t\le0}=\{\Phi^*_t\}_{t\ge0} $ (the proof is omitted). However, there is not a group structure on $\{\Phi^t\}_{t\in\RR}$. As an example, consider $\Phi^6 \in G_1$ and $\Phi^{-1}\in G_2$, and let $x=-4$:
	\[ \Phi^6 (\Phi^{-1}(-4))=[0,9/4], \hspace{.2in} \Phi^5(-4)=[0,1/4],\hbox{ and }\hspace{.2in} \Phi^{-1}(\Phi^6(-4))=[-1/4,1/4].\]
\end{rmk}

Multiflows can also arise, with a simple restriction of a domain.

\begin{ex}\label{ex:DispRestrictDomain}
We offer an example of a simple ODE: $\dot{x} = 1$. When the space $X$ is $\RR$, we get a flow: $\vp(t,x_0)=x_0+t$. Looking at the time $t$ maps: $f:X\rarrow X$, we get $f(x)=x+t$, which are graphed as all the lines with slope 1.  

Now, let us restrict to $X=[-1,1]$ and use the same differential equation. First of all, the graph of $X\times X$ is now a box of side length 2. It actually does not make sense to define a flow because, for instance, there is no image of $x_0=1$ for time $t>0$. Even with $x_0=-1$, the orbit does not last past $t=2$. Similar issues exist for negative $t$ values. The graph is no longer of lines, but of line segments.

\begin{figure}\centering
  \includegraphics[width=.4\linewidth]{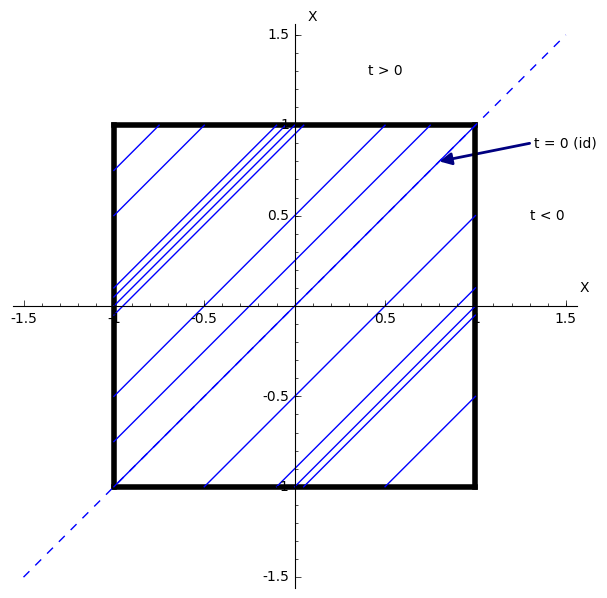}
  \caption{Image of Multiflow from Example \ref{ex:DispRestrictDomain}.}
  \label{fig:BifnExRestDom}
\end{figure}

\FloatBarrier

This is not a very complicated example, but we cannot truly apply the notation of flows and maps to it. So, we define a multiflow. Let $\Phi\subset \RR\times X\times X$ to be
	\[\Phi=\{(t,x,y):~y=x+t, \hbox{ with } x,y\in X\}.\]
We actually end up with restrictions on $t$, depending on $x\in X$: $-1\le t+x\le 1,$ or $-2\le-1-x\le t\le 1-x\le 2.$

Again, we are able to discuss a multiflow, which is defined for some negative numbers of $t$, but again the negative values will agree with $\Phi^*$.

\end{ex}

\newpage
\section{Ultimate Behavior and Invariance}\label{sec:UltBehavior}

We are often concerned with the ultimate behavior of a system. Given a set, we wish to know what will happen if we follow it ``forever" in forward or backward time. Ideally, we come up with a set, which will not change if acted upon again - we want to find an {\it invariant} set. These sets (and the gradient-like flows between them) make up the structure of our space, if we are acting by flows, maps \ref{bk:Conley}, or relations \ref{bk:McGeheeAttractors}. This is the notion of the {\it omega limit set}, or $\om$-limit set. A similar notion - the alpha limit set ($\alpha$-limit set) - exists for backward time (when backward time is defined). We will first define what we mean by ``invariant."

\begin{defn}\label{forinvflow}
Let $\phi$ be a flow over a space $X$, and let $S\subset X$ be a subset such that $\phi^t(S) = \{ \phi^t(s):s\in S\}\subset S$ for all $t>0$. Then $S$ is said to be {\bf forward invariant} (or positively invariant).
\end{defn}

\begin{defn}\label{ForInvMap}
A set $S\subset X$ is {\bf forward invariant} (or positively invariant) under map $f$ if $f(S)\subset S$.
\end{defn}

It is natural to simply reverse the time in the flow definition above, and gain the definition for a {\it backward invariant} set. This is equivalent to the following.

\begin{defn}
The set $S\subset X$ is {\bf backward invariant} for the flow $\phi$ on $X$ if $S\subset \phi^t(S)$ for all $t\ge0$. 
\end{defn}

For maps, we are not guaranteed that $f^{-1}$ exists. Hence the need to add a condition to the following.

\begin{defn}
If $f$ is a bijective map over a space $X$, then $S\subset X$ is called {\bf backward invariant} if $f^{-1}(S)\subset S$.
\end{defn}

\begin{defn}Let $\phi$ be a flow over a space $X$.
The set $S$ is {\bf invariant} for $\phi$ if
	\[\phi^t(S)=S, \hspace{.5in}\hbox{ for all }t\in\RR,\]
	or, equivalently, if $S$ is both forward and backward invariant.
\end{defn}

\begin{defn}
Given a map $f$ on a space $X$,
we call $S\subset X$ {\bf invariant} under $f$ if $f(S)=S$. If $f$ is invertible, this is equivalent to $f^{-1}(S)=S$.
\end{defn}


\begin{defn}\label{defn:OmLimSet}
Let $f:X\rightarrow X$ be a map and let $S\subset X$. The {\bf omega limit set} of $S$ is defined as 
	\[\om(S;f)\equiv\bigcap\limits_{m\ge 0}\ol{\bigcup\limits_{n\ge m} f^n(S) }.\]
If $f^{-1}$ is defined, then the {\bf alpha limit set} of $S$ is
	\[\alpha(S;f)\equiv\bigcap\limits_{m\le 0}\ol{\bigcup\limits_{n\le m} f^n(S) }.\]
\end{defn}

Note that the alpha limit set is just the omega limit set for the map$f^{-1}$.

\begin{defn}
Let $\phi$ be a flow (or semiflow) on a space $X$ and let $S\subset X$. The {\bf omega limit set} of $S$ is defined as 
	\[\om(S;\phi)\equiv\bigcap\limits_{t\in[0,\infty)}\ol{\phi^{[t,\infty)}(S) }.\]
If $\phi$ is a flow, then $\phi^{-1}$ is defined, and the {\bf alpha limit set} of $S$ is
	\[\alpha(S;\phi)\equiv\bigcap\limits_{t\in(-\infty,0]}\ol{ \phi^{(-\infty,t]}  }.\]
\end{defn}

Omega limit sets are guaranteed to be invariant. The closures in the above definitions are necessary, as without them the resulting sets are not always invariant. In defining omega limit sets for relations (see \ref{bk:McGeheeAttractors}), the goal was to find an invariant set that, when restricted to the class of maps, agrees with the Definition \ref{defn:OmLimSet}. In defining omega limit sets for multiflows, we have a similar goal.

\begin{ex}\label{ex:OmMap}
To gain insight on how $\om(S;f)$ is calculated, we return to Example \ref{ex:mapPerOrbit}. Specifically, we look at $f^2$, $f^3$, and $f^4$, which - together - also happen to illustrate all the larger iterations of $f.$
Thus, $\bigcup\limits_{n\ge k} f^n$ will always contain maps, which look like $f^2$, $f^3$, and $f^4$. If we follow $S=\{a,b\}$ for instance, $f(S)=\{b,c\}$. If we keep going, we get $f(S)\cup f^2(S)\cup f^3(S)\cup\ldots = \{b,c,d,e\}$, and because our set is discrete, this is the same as the closure. But, we begin to delete previous iterates: $\ol{f^2(S)\cup f^3(S)\cup f^4(S)} = \{c,d,e\}$ because eventually (almost immediately) $b$ was no longer in the image of $S$ or the closure of the union. In this example, we don't lose any more points, and so $\om(S;f)=\{c,d,e\}$.


\begin{figure}
    \centering
    \subfigure[$n\equiv 2$ mod $3$]
    {
        \includegraphics[width=.2\linewidth]{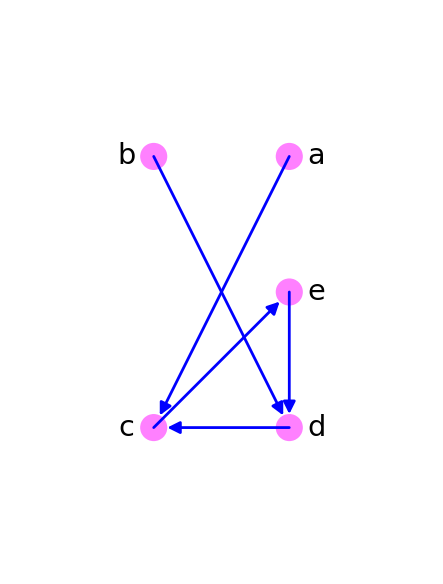}
        \label{fig:MapEx1f2Mod}
    }
    \subfigure[$n\equiv 0$ mod $3$]
    {
        \includegraphics[width=.2\linewidth]{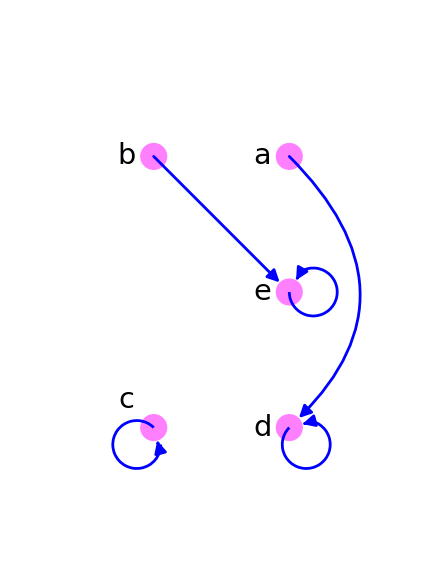}
        \label{fig:MapEx1f0Mod}
    }
    \subfigure[$n\equiv 1$ mod $3$]
    {
        \includegraphics[width=.2\linewidth]{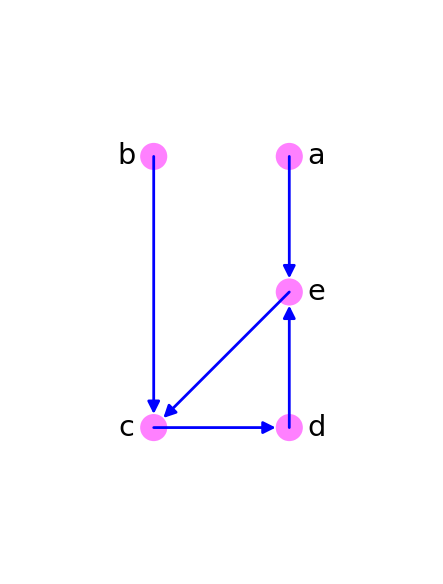}
        \label{fig:MapEx1f1Mod}
    }
    \caption{All illustrations of $f^n$, $n>1$.}
    \label{fig:MapEx1Iterations}
\end{figure}

\begin{figure}\centering
	\includegraphics[width=1.5in]{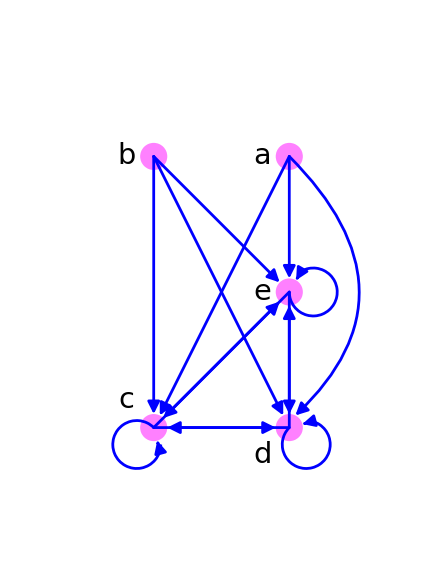}
	\caption{Illustration of $\ol{\bigcup\limits_{n\ge m}f^n }$ from Examples \ref{ex:mapPerOrbit} and \ref{ex:OmMap}.}\label{fig:MapEx1fInfty}
\end{figure}
\FloatBarrier

\end{ex}

\subsection{Ultimate Behavior and Invariance for Relations}\label{sec:UltBehaviorRelns}

Recall that when $f:X\rarrow X$ is an invertible map and $S\subset X$,
	\[   f(S)\subset S \iff S\subset f^{-1}(S),\]
and we refer $S$ as being {\it forward invariant}. With relations, where $f^{-1}$ is usually not even defined, it makes sense to split these notions. Thus, in \ref{bk:McGeheeAttractors}, the following terms were assigned.

\begin{defn}\label{defn:invariance} If $f$ is a relation over $X$ and $S\subset X$, then we use the following to describe $S$:
	\begin{enumerate}
		\item If $f(S)\subset S$, then $S$ is called a {\bf confining} set for $f$;
		\item If $f^*(S)\subset S$, then $S$ is called a {\bf rejecting} set for $f$;
		\item If $f(S)\supset S$, then $S$ is called a {\bf backward complete} set for $f$;
		\item If $f^*(S)\supset S$, then $S$ is called a {\bf forward complete} set for $f$;
		\item If $f(S) = S$, then $S$ is called an {\bf invariant} set for $f$;
		\item If $f^*(S) = S$, then $S$ is called a {\bf *-invariant} set for $f$;
	\end{enumerate}
\end{defn}

We use these new terms in part so as not to be tempted to assume the same structures that exist with maps. 

\begin{ex}
We bring to mind again Example \ref{ex:mapPerOrbit}, in which $X=\{a,b,c,d,e\}$, and $f(a)=b$, $f(b)=c=f(e)$, $f(c)=d$, and $f(d)=e$. The set $S=\{c,d,e\}$ is an orbit and an invariant set because $f(S)=\{c,d,e\}=S$.

Unlike with an invertible map, the inverse image of $c$ is a set with more than one element: $f^{-1}(c) = \{b,e\}$. We actually escape the periodic orbit in backward time. So, $f^{-1}(S)=f^{-1}(\{c,d,e\}=\{b,c,d,e\}\neq S.$ The set $\{c,d,e\}=S$ certainly defines a periodic orbit, and therefore an invariant set with respect to $f$, but $S$ is not backward invariant with respect to $f$.
\end{ex}

\begin{thm}\label{thm:RelnInvcUandInt}
If $f$ is a relation over a space $X$, and $\frak{A}$ is a set of invariant sets $A\subset X$. Then,
\begin{enumerate}
	\item $\bigcup\limits_{A\in\frak{A}}A$ is invariant, and
	\item $\bigcap\limits_{A\in\frak{A}}A$ is confining.
\end{enumerate}
\end{thm}

\begin{proof} Let $f$, $X$, and $\frak{A}$ be as described.
(1) We apply Lemma \ref{lem:relnUandInt}:
	\[    f\left(\bigcup\limits_{A\in\frak{A}}A\right)=\bigcup\limits_{A\in\frak{A}}f(A)=\bigcup\limits_{A\in\frak{A}}A.\]

(2) We again apply Lemma \ref{lem:relnUandInt}, but for intersections, we only get inclusion:
	\[    f\left(\bigcap\limits_{A\in\frak{A}}A\right)\subset\bigcap\limits_{A\in\frak{A}}f(A)=\bigcap\limits_{A\in\frak{A}}A.\]
\end{proof}

\begin{ex}\label{ex:RelnNotPresIntnEvenInvt}
Example \ref{ex:RelnsDoNotPresIntn} actually gave us an example, where intersections weren't preserved, even for a {\it finite} intersection of invariant sets.
\end{ex}

\begin{ex}[borrowed from Example 5.2 in \ref{bk:McGeheeAttractors}]\label{ex:ClosureofConfNotConf}
Let $X=[-2,+\infty]$ (the one point compactification of $[-2,+\infty)$). Consider the continuous surjective map (and therefore closed relation)
	\[       g(x) = \left\{\begin{array}{ll}
				x^2-2, & \hbox{if }x\in[-2,+\infty) \\
				+\infty, & \hbox{if }x=+\infty\end{array}\right. .\]
Let $f=g^*$ (the transpose of $g$). Consider the set $S=S=(2,\eta]$, then $f(S)=(2,\sqrt{\eta+2}]\subset S$, but $f(\ol{S})=[2,\sqrt{\eta+2}]\cup\{-2\}\not\subset\ol{S}.$ That is, $S$ is confining under $f$, but $\ol{S}$ is not confining. We refer to the actual purpose of the original Example 5.2 in a remark later in this section.
\end{ex}

The simplest situation is when $S$ is confining (or rejecting, or invariant...), but it is also useful to distinguish between the following.

\begin{defn}
Let $f$ be a relation over $X$. Then for $U\subset X$,
\begin{enumerate}
	\item $U$ is {\bf confining} for $f$ if $f(U)\subset U$.
	\item $U$ is {\bf strictly confining} for $f$ if $f(U)\subset \Int(U)$.
	\item $U$ is {\bf eventually confining} for $f$ if there exists some $N>0$ for which $f^n(U)\subset U$ for all $n\ge N$.
	\item $U$ is {\bf eventually strictly confining} for $f$ if there exists some $N>0$ for which $f^n(U)\subset \Int(U)$ for all $n\ge N$.
\end{enumerate}
\end{defn}

Instead of ``confining," we may use the term {\it immediately confining} to emphasize that it is stronger than {\it eventually confining}, but the immediacy is implied, whenever we omit the word ``eventually."

\begin{lem}\label{lem:evConfUnion}
Let $f$ be a closed relation, and let $E$ be a closed eventually confining set under $f$. Let $n$ be a number such that $f^n(E)\subset E$. Let \[B=\ol{\bigcup\limits_{k=0}^{n-1} f^k(E)}=\bigcup\limits_{k=0}^{n-1} f^k(E)=E\cup f(E)\cup f^2(E)\cup \ldots \cup f^{n-1}(E).\] Then $B$ is a closed, confining set.
\end{lem}

\begin{proof}
The union is already closed, because it is a finite union of closed sets.
Let $x\in f(B)=\bigcup\limits_{k=1}^{n} f^k(E)$. Then $x\in f^m(B)$ for some $1\le m\le n$. If $m<n$, then $x$ is explicitly in $B$. If instead $m=n,$ then $x\in f^n(E)\subset E$. Thus $B$ is confining.
\end{proof}

It is worth noting that we cannot simply take a similar union with continuous time and be guaranteed a closed set, as the union over continuous time is not a finite union. This provided additional complications for multiflows. We also cannot just take the closure of such a union. As Example \ref{ex:ClosureofConfNotConf},
	\[ S \hbox{ is confining }\centernot\implies \ol{S} \hbox{ is confining.}\]

\begin{ex}\label{ex:evStrConfReln}
Let $\phi$ be the flow on $X=\RR\times\RR$, defined by $\dot{x}=-x-y$ and $\dot{y}=x-y$. Choose $S$ to be the closed set between the curves pictured (with parametric curve $x=e^{-t}\cos(t),$ $y=\pm e^{-t}\sin(t)$ for $t\in[0,\pi]$).

\begin{figure}[h]\centering\label{fig:ExampleConfiningEventStrictConf}
  \includegraphics[width=2.5in]{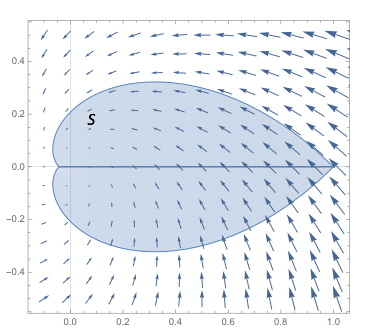}
  \caption{Image of $S$, which is confining and eventually strictly confining.}
\end{figure}

We will revisit this example later - in Example \ref{ex:evStrConf}, where $S$ will serve as an (immediately) confining, but only eventually strictly confining for multiflows. For now, let us discuss the fixed time relations $\phi^t$ (they're actually fixed time maps). For $0<t<\pi$, these relations are (immediately) confining, but only eventually strictly confining.
\end{ex}

Equivalent distinctions may eventually prove useful for the other terms in Definition \ref{defn:invariance} (rejecting, etc.), but we will not use them.
Note that a set that is confining is also eventually confining, a set which is strictly confining is confining, and a set which is eventually strictly confining is eventually confining. The proofs are omitted.

\[ \begin{array}{ccc}  S\hbox{ strictly confining} &  \implies    &   S\hbox{ confining}  \\ 
                                \rotatebox[origin=c]{270}{$\implies$} &   & \rotatebox[origin=c]{270}{$\implies$}        \\
                                S\hbox{ eventually strictly confining}        &   \implies             &S\hbox{ eventually confining}     \end{array} \]

Equipped with notions of confining and invariant sets, we wish to talk about ultimate behavior, which means we need the definition of an omega limit set.
For relations, we actually need a different definition than for maps. We call the one, which corresponds to the map definition, the strict omega limit set.

\begin{defn} The {\bf strict omega limit set} of a set $S\subset X$ under the relation $f$ is
	\[ \hat{\om}(S;f)\equiv \bigcap\limits_{m\ge 0}\ol{\bigcup\limits_{n\ge m} f^n(S) }.\]
\end{defn}

\begin{defn}\label{defn:OmegaReln} If $f$ is a relation over $X$ and $S\subset X$, then the {\bf omega limit set} of $S$ under $f$ is
	\[ \om(S;f)\equiv \bigcap \frak{K}(S;f),\]
where 
	\begin{align*} \frak{K}(S;f)\equiv& \{ K: K \hbox{ is a closed confining set satisfying}\\
							& f^n(S)\subset K \hbox{ for some }n\ge0\}.\end{align*}
We may abbreviate $\konf(S)$, $\konf_f$ or even $\konf$ when the $S$ or $f$ is understood from context.
\end{defn}

A similar notion - the {\it alpha limit set} - will not be explored in depth in this paper, but just like any other backwards-time definitions, it would use the transpose of a relation. See \ref{bk:McGeheeAttractors} for more.

\begin{rmk}\label{rmk:OmNotOmHatReln}
See Example 5.2 in \ref{bk:McGeheeAttractors} for an example, which explicitly shows not only that $\om(S;f)$ and $\hat{\om}(S;f)$ can differ, but also that sometimes the strict omega limit set is not invariant. For multiflows, we will explore a different example.
\end{rmk}

The strict omega limit set is easier to compute (and more familiar in design), and in many instances the strict omega limit and the omega limit agree. In particular, they must agree when $f$ is a map, and not just a relation. The following theorems prove invaluable in any groundwork on multiflows.

\begin{thm}[Theorem 5.1 in \ref{bk:McGeheeAttractors}]\label{thm:InclOmLimReln}
If $f$ is a closed relation on a compact Hausdorff space $X$ and $S\subset X$, then $\hat{\om}(S)\subset\om(S)$.
\end{thm}

There are no special properties of $S$; this inclusion holds for omega limit sets on all subsets of $X$.

\begin{thm}[Theorem 5.3 in \ref{bk:McGeheeAttractors}]\label{thm:EquivOmLimRelnMap}
If $g$ is a continuous map on a compact Hausdorff space $X$ and $S\subset X$, then
		\[ \om(S;g)=\hat{\om}(S;g) \]
\end{thm}

This was a valuable result, as it lent weight to the notion that Definition \ref{defn:OmegaReln} is the correct object to consider. We will contend with resolving even more definitions when we encounter multiflows.

\begin{thm}[Theorem 5.4 in \ref{bk:McGeheeAttractors}]\label{thm:EquivOmLimRelnClConf}
If $S$ is a closed confining set for a closed relation $f$ on a compact Hausdorff space, then
	\[ \hat{\om}(S)=\om(S)=\bigcap\limits_{n\ge0}f^n(S).\]
\end{thm}

It turns out $S$ does not even need to be confining. It is enough that a subset is eventually confining.

\begin{thm}\label{thm:EvConfRel}If $f$ is a closed relation over a compact Hausdorff space $X$ and $U\cl{\subset} X$ is {\em eventually confining} with respect to $f$, then
	\[\hat{\om}(U;f)= \om(U;f).\]
\end{thm}

\begin{proof}
Let $U\cl\subset X$ where $X$ is a compact Hausdorff space, $f$ a closed relation on $X$, and $U$ is eventually confining with respect to $f$. Recall that it is always true that $\hat{\om}(U;f)\subset \om(U;f)$. Let $N$ be a positive integer such that $f^n(U)\subset U$ for all $n\ge N$. Set
	\[V_0=\bigcup\limits_{0\le k <N} f^k(U)=U\cup \bigcup\limits_{1\le k< N}f^k(U)\hbox{ and }\mathfrak{V}=\{f^k(V_0)\}_{k\in\ZZp}\]
Note that $V_0$ is closed (a finite union of closed sets), as are all $V\in\mathfrak{V}$ (because $X$ is compact). Also, $V_0$ is confining:
	\[f(V_0)=\left(\bigcup\limits_{1\le k<N}f^k(U)\right)\cup f^N(U)\subset V_0\cup U\subset V_0,\]
and we have already established that the image of a confining set under a closed relation is confining, so every $V\in\mathfrak{V}$ is confining. Furthermore, $f(U)\subset V_0$, so $f^{k+1}(U)\subset f^k(V_0)$. Therefore, $\mathfrak{V}\subset\konf$, where $\konf$ is as defined in Definition \ref{defn:OmegaReln}:
	\begin{align*}\konf_f = \{& K : K\hbox{ is a closed confining set such that there exists some }n\ge0 \\
				  &\hbox{ for which }f^n(U)\subset K\}.\end{align*}
Because $\mathfrak{V}\subset\konf$, $\bigcap\limits_{V\in\mathfrak{V}}V\supset\bigcap\limits_{K\in\konf}K=\om(U;f)$. 

We wish to show that $\bigcap\limits_{V\in\mathfrak{V}}V\subset \hat{\om}(U;f)=\bigcap\limits_{n\ge0} \overline{\bigcup\limits_{k\ge n} f^k(U)}$. Let an arbitrary $\overline{\bigcup\limits_{k\ge n} f^k(U)}$ be given. Because $f^{l+N}(U)=f^l(f^N(U))\subset f^l(U),$
	\begin{align*}	\overline{\bigcup\limits_{k\ge n} f^k(U)}=&\overline{f^n(U)\cup f^{n+1}(U)\cup\ldots\cup 						f^{n+N-1}(U)\cup\ldots}\\
					=&\overline{f^n(U)\cup f^{n+1}(U)\cup\ldots\cup f^{n+N-1}(U)} \\
					=&f^n(U)\cup f^{n+1}(U)\cup\ldots\cup f^{n+N-1}(U)\\
					=&f^n\left(U\cup f(U)\cup \ldots f^{N-1}(U)\right)\\
					=&f^n(V_0)\in\mathfrak{V}
\end{align*}
Thus, $\hat{\om}(U;f)=\bigcap\limits_{V\in\mathfrak{V}}V\supset\bigcap\limits_{K\in\konf}K=\om(U;f).$
\end{proof}

The following do not really follow a narrative flow, outside of \ref{bk:McGeheeAttractors}, but they will also be used in Sections \ref{sec:OmLimSetsDisp} and \ref{sec:AttAndAttBlk}

\begin{lem}[Lemma 5.5 in \ref{bk:McGeheeAttractors}]\label{thm:KappaInclReln}
If $S$ is a subset of $X$ with $f$ and $g$ relations over $X$ then the following properties hold.
\begin{enumerate}
\item $S'\subset S\implies \frak{K}(S;f)\subset\frak{K}(S';f).$
\item $g\subset f\implies \frak{K}(S;f)\subset\frak{K}(S;g).$
\item $\frak{K}(f(S);f)=\frak{K}(S;f).$
\item $K\in\frak{K}(S;f)\implies f(K)\in\frak{K}(S;f).$
\item If $\frak{F}$ is a finite subset of $\frak{K}(S;f),$ then $\bigcap\frak{F}\in\frak{K}(S;f).$
\end{enumerate}
\end{lem}

\begin{cor}[Corollary 5.6 from \ref{bk:McGeheeAttractors}]
If $f$ is a relation over $X$, $S\subset X$, and $U$ is a neighborhood of $\om(S)$, then there exists a $K\in \frak{K}(S;f)$ such that $K\subset U$.
\end{cor}

\begin{lem}[Lemma 7.6 from \ref{bk:McGeheeAttractors}]\label{thm:RelnsOmCandInt}
If $B$ is an attractor block ($f(\ol{B})\subset\Int(B)$), then the following statements are true. 
\begin{enumerate}
	\item $B$ is a confining set.
	\item $\Int(B)$ and $\ol{B}$ are attractor blocks.
	\item $\om(\Int(B)) = \om(B) = \om(\ol{B})\subset \ol{B}$.
\end{enumerate}
\end{lem}

In studying the structure of a dynamical system, it is useful to look at isolated invariant sets. More precisely, we identify the {\it isolating neighborhoods} associated to those invariant sets, but the invariant sets themselves are the motivation. 
To properly define isolated invariant sets, in general, we need to understand what it is to be a {\it maximal invariant set}.


\begin{defn}\label{maxinvset}
An invariant set can also be {\bf maximal}, relative to the set containing it. Let $U\subset X$, where $X$ is the topological space. Let $S\subset U$ be an invariant set. Then $S$ is maximal if there is no invariant set $T$ such that $S\subsetneq T\subset U$.
\end{defn}

\begin{defn}\label{defn:IsoInvSet}
A compact set $S$ in a space $X$ is an {\bf isolated invariant set} if it is invariant and if there exists some neighborhood $U$ of $S$ in which $S$ is the maximal invariant set. In this case, $U$ is an {\bf isolating neighborhood} of $S$
\end{defn}

Particular examples include attractors (which are the omega limit set of a neighborhood $U$), repellers (which are the alpha limit set of a neighborhood $U$), and saddle points. An isolating neighborhood will have only one maximal invariant set, but an isolated invariant set may have many isolating neighborhoods.


\begin{ex}\label{ex:NoIso}
Not all invariant sets - even compact ones - can be isolated. For instance, let $\vp: \RR\times \RR^2 \rarrow \RR^2$ be the solution to
	\[ \left[\begin{array}{c} \dot{x} \\ \dot{y} \end{array}\right] = \left[\begin{array}{cc} 0 & 1 \\ -1 & 0 \end{array}\right]\left[\begin{array}{c} x \\ y \end{array}\right],\] 
as in Example \ref{ex:rotnPerOrb} (a rotation with period $2\pi$ about the origin). Then,
	\[\vp^t(x_0,y_0)= \left[\begin{array}{cc} \cos t & \sin t \\ -\sin t & \cos t \end{array}\right]\left[\begin{array}{c} x_0 \\ y_0 \end{array}\right]\] 
for any $(x_0,y_0)\in\RR^2$. See Figure \ref{fig:rotnPerOrb} for a visual representation of a few orbits.

Let $S$ be any nonempty compact invariant set for $\vp$. In this system, any compact invariant set is a union of periodic orbits. Then, $S\subset B_r(0,0)$ for some $r\ge0$. Note that any neighborhood $U$ of $S$ will contain some larger closed ball $\overline{B_s}(0,0)$ (with $s>r$). This new ball contains orbits, which were not in $S$. Explicitly, $\bigcup\limits_{t\in\RR}(s\cos t, s\sin t)\subset \overline{B_s}(0,0)\setminus S$, making $\overline{B_s}(0,0)$ a new, larger invariant set in $U$. Therefore, $S$ is not maximal in $U$, and it cannot be isolated.

\begin{figure}\centering
  \includegraphics[width=.4\linewidth]{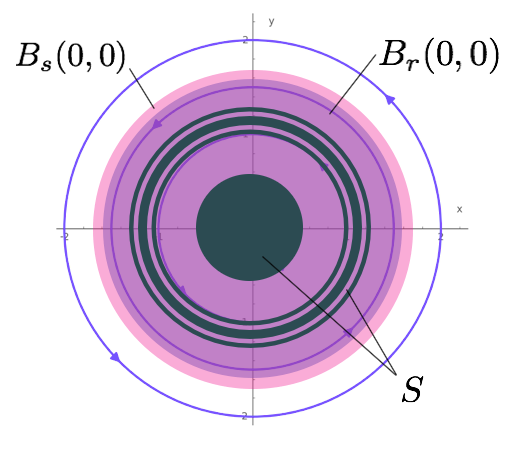}
  \caption{Illustration of $S$ and the surrounding balls for Example \ref{ex:NoIso}.}
  \label{fig:NoIso}
\end{figure}

\FloatBarrier
\end{ex}

\begin{defn}\label{defn:Att} Let $\phi$ (or $f$) be a flow (or map) over $X$. Then, $A\subset X$ is an {\bf attractor}, provided it is compact, and
	\begin{itemize}
		\item $A$ is invariant with respect to $\phi$ (or $f$),
		\item $A=\om(B)$ for some neighborhood $B$ of $A$.
	\end{itemize}
\end{defn}

The neighborhood spoken of in Definition \ref{defn:Att} is an example of an isolating neighborhood, or more specifically, an {\bf attracting neighborhood} for $A$. This is actually one definition of an attracting neighborhood, and one which we will use to construct similar definitions for multiflows.

\begin{defn}
	Let $\phi$ (or $f$) be a flow (or map) over $X$, and let $B\subset X$ be a compact set such that $\om(B)\subset \Int(B)$. Then $B$ is an {\bf attracting neighborhood}. Moreover, we say $B$ is an attracting neighborhood associated to $\om(B)$, and $\om(B)$ is the attractor associated with $B$.
\end{defn}

Another way to talk about attractors is to speak of {\it Lyapunov stable} and {\it asymptotically stable} sets, but we will not explore those ideas here.

One of the reasons we consider attracting neighborhoods is that for maps and flows there is a continuation property, and an upper-semicontinuity. That is, if we have a flow $\phi$ with an isolating neighborhood $N$ (attracting neighborhoods are a type of isolating neighborhood), and we consider any flow $\psi$ which is ``close enough" to $\phi$, then $N$ serves as an isolating neighborhood for $\psi$.

\begin{ex}\label{ex:InvSetFlow}
Let $x\in \RR$ and define $\vp$ to be a solution to
\[\dot{x}=(x-1)(x+1) = x^2-1.\]
So, $\vp(t,x)=\frac{ke^{2t+1}}{1-ke^{2t}}$ for any $k\in \RR$. The minimal invariant sets are the equilibria $S_1=\{-1\}$ and $S_2=\{1\}.$ In fact, $S_1$ is a sink (an attracting equilibrium), and $S_2$ is a source (a repelling equilibrium). We can choose an attracting neighborhood $B = [-2,0]$, which contains the attractor $S_1$. Let's call $\dot{x}_\ep = x^2-1+\ep$, and so long as $|\ep|<1$, $B$ is still an attracting neighborhood containing a sink, which behaves similarly to $S_1$. That is, $B$ remains an attracting neighborhood for systems near $\dot{x}=x^2-1$.
\end{ex}

We now generalize these notions to cover invariant sets, which are not in general attractors.

\begin{defn}\label{defn:inv}
Let $\vp$ be a flow and $f$ be a map on $X$, with $N\subset X$. Then, the {\bf maximally invariant subset of $N$} is defined in the following ways:
\[\Inv(N,\vp):=\{x\in N:~\phi(\RR,x)\subset N\},\]
and likewise
\[\Inv(N,f):=\{x\in N:~ f^n(x)\subset N,~n\in\NN\}.\]
\end{defn}



\begin{defn}
A nonempty compact set $N\subset \RR^n$ is an {\bf isolating set / isolating neighborhood} if
\[\Inv(N,\vp)\subset N^\circ.\]
\end{defn}

\begin{defn}\label{isoblock}
Let $S$ be a compact invariant set. An {\bf isolating neighborhood associated to $S$} is a set $N$ such that
\[S\subset N^\circ, ~ \Inv(N)= S.\]
If such a neighborhood exists, $S$ is called an {\bf isolated invariant set}.
\end{defn}

An attractor is a particular type of invariant set, and an attracting neighborhood is a particular type of isolating neighborhood. 

\subsection{Confining and Invariant Sets for Multiflows}

Recall that a set $U$ is confining under relation $f$ if $f(U)\subset U$. This notion is similar to forward invariance for maps. We now extend this and related notions to continuous time. Naturally, the terms in Definition \ref{defn:invariance} already carry over to any particular fixed time relation.

\begin{defn}\label{defn:DispInvariance} If $\Phi$ is a multiflow over $X$ and $S\subset X$, then we use the following to describe $S$:
	\begin{enumerate}
		\item If $\Phi^t(S)\subset S$ for all $t\ge0$, then $S$ is called a {\bf confining} set for $\Phi$;
		\item If $(\Phi^t)^*(S)\subset S$ for all $t\ge0$, then $S$ is called a {\bf rejecting} set for $\Phi$;
		\item If $\Phi^t(S)\supset S$ for all $t\ge0$, then $S$ is called a {\bf backward complete} set for $\Phi$;
		\item If $(\Phi^t)^*(S)\supset S$ for all $t\ge0$, then $S$ is called a {\bf forward complete} set for $\Phi$;
		\item If $\Phi^t(S) = S$ for all $t\ge0$, then $S$ is called a {\bf invariant} set for $\Phi$;
		\item If $(\Phi^t)^*(S) = S$ for all $t\ge0$, then $S$ is called a {\bf *-invariant} set for $\Phi$;
	\end{enumerate}
\end{defn}

For discussion on why these terms were chosen for relations, refer to \ref{bk:McGeheeAttractors}. We also have the expanded notions of eventually confining, strictly confining, etc.

\begin{defn}
Let $\Phi$ be a multiflow over $X$. Then for $U\subset X$,
\begin{itemize}
	\item $U$ is {\bf confining} if $\Phi^t(U)\subset U$ for all $t>0.$
	\item $U$ is {\bf strictly confining} if $\Phi^t(U)\subset \Int(U)$ for all $t>0.$
	\item $U$ is {\bf eventually confining} if there exists some $T>0$ for which $\Phi^t(U)\subset U$ for all $t\ge T$.
	\item $U$ is {\bf eventually strictly confining} if there exists some $T>0$ for which $\Phi^t(U)\subset \Int(U)$ for all $t\ge T$.
\end{itemize}
\end{defn}

It is possible to have a set, which is immediately confining, but only eventually strictly confining, as this example demonstrates.

\begin{ex}\label{ex:evStrConf}
Let $X=\RR^2$, and let
	\[ \dot{\left[ \begin{array}{c}x \\ y\end{array}\right]} = \left[ \begin{array}{cc} -1 & -1 \\ 1 & -1 \end{array}\right] \left[ \begin{array}{c}x \\ y\end{array}\right] \]
define a vector field, as in Example \ref{ex:evStrConfReln}. The solutions are of the form 
	\[ \left[ \begin{array}{c}x(t) \\ y(t)\end{array}\right] = e^{-t}\left[ \begin{array}{cc}\cos t & -\sin t\\ \sin t & \cos t \end{array}\right] \left[ \begin{array}{c}x_0 \\ y_0\end{array}\right],\]
where $\left[ \begin{array}{c}x_0 \\ y_0\end{array}\right]$ is the initial point.

Take $S$ to be the closed ``heart-shaped" region bounded by the curves $\left[ \begin{array}{c} x\\y \end{array}\right] = e^{-\theta} \left[ \begin{array}{c} \cos \theta \\ \pm\sin \theta\end{array}\right]$, pictured in Figure \ref{fig:RegionConfEvStrConf}. This region will be confining for $0\le t$, but it is only strictly confining for $t>\pi$. Thus, $S$ is confining and eventually strictly confining, but it fails to be strictly confining.

\begin{figure}
    \centering
    \subfigure[A look at the whole set described in Example \ref{ex:evStrConf}.]
    {
        \includegraphics[width=.35\linewidth]{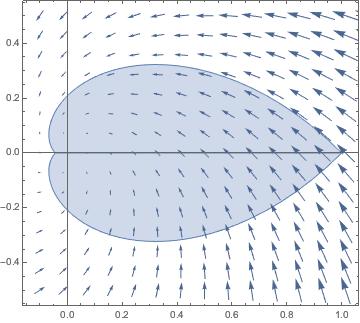}
        \label{fig:ConfEvStrConf}
    }
    \subfigure[A view nearer the origin in Example \ref{ex:evStrConf}.]
    {
        \includegraphics[width=.35\linewidth]{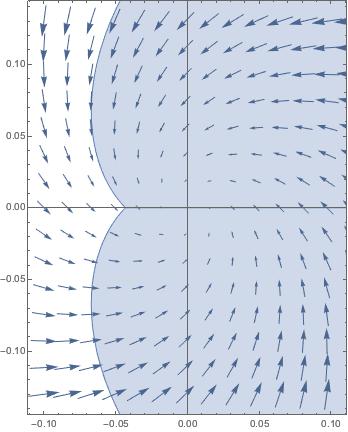}
        \label{fig:ConfEvStrConfZoom}
    }
    \caption{A set which is confining but only eventually strictly confining.}
    \label{fig:RegionConfEvStrConf}
\end{figure}
\end{ex}

\begin{ex}
Example \ref{ex:NoIso} gives a vector field for which any closed circular region centered at the origin is confining (and in fact invariant) but never strictly confining.
\end{ex}

It is good to know how the properties of a set under a multiflow carry over to the properties of that set under the fixed time relations, and how this affects forward images, as Theorems \ref{ConfDispConfRel} and \ref{confAndImages} address.

\begin{thm}\label{ConfDispConfRel}
Let $\Phi$ be a multiflow over a space $X$ and $U\subset X$. Then the following properties hold true for any fixed time relations $\Phi^t$, $t>0$:
\begin{enumerate}
\item $U$ is confining for $\Phi\implies U$ is confining for $\Phi^t$;
\item $U$ is strictly confining for $\Phi\implies U$ is strictly confining for $\Phi^t$;
\item $U$ is eventually confining for $\Phi\implies U$ is eventually confining for $\Phi^t$;
\item $U$ is eventually strictly confining for $\Phi\implies U$ is eventually strictly confining for $\Phi^t$.
\end{enumerate}
\end{thm}

\begin{proof}

Let $U\subset X$ and $\Phi$ be as described in the theorem. We start with (4) and assume that $U$ is eventually strictly confining in the multiflow sense for $\Phi$. So, there is some $\tau\in(0,\infty)$ such that $\Phi^t(U)\subset \Int(U)$ for all $t\ge\tau$. Let $f=\Phi^t$ for some fixed $t\in(0,\infty)$. 
We find $N$ such that $Nt\ge \tau$. Then, for all $n\ge N$,
	\[f^n(U)=\Phi^{nt}(U)\subset \Int(U)\]
	because $nt\ge Nt\ge \tau$.
	
The other cases follow suit, mutatis mutandis.
\end{proof}

In fact, similar notions are true for the various classifications: if $U$ is eventually rejecting for $\Phi$, then $U$ is eventually rejecting for each $\Phi^t$, etc. 

\begin{thm}\label{confAndImages}
Let $\Phi$ be a multiflow over a space $X$. Then, consider $U\subset X$ and all forward images $\Phi^s(U)$, where $s\ge0$. The following are true:
	\begin{enumerate}
		\item If $U$ is confining for $\Phi$, then $\Phi^s(U)$ is confining for $\Phi$.
		\item If $U$ is eventually confining for $\Phi$, then $\Phi^s(U)$ is eventually confining for $\Phi$.
	\end{enumerate}
\end{thm}

\begin{proof}
	Let $U$ be (eventually) confining. Then, $\Phi^t(U)\subset U$ for all $t\ge0$ (or $t\ge T$ for some $T>0$ in the case of {\it eventually} confining). Consider the forward image $\Phi^s(U)$ for $s>0$. Then,
		\begin{align*}
			\Phi^t(\Phi^s(U)) = \Phi^{t+s}(U) = \Phi^s(\Phi^t(U))\subset \Phi^s(U),
		\end{align*}
for $t\ge0$ (or $t\ge T$ in the case of eventually confining).
\end{proof}

\subsection{Defining the $\konf$-Sets for Multiflows}

In creating omega limit sets for relations, we needed to define $\konf(U;f)$. The same is true for multiflows. We actually have two sets to define - one which functions on fixed-time relations, and one which links to multiflows in a more natural way. The former need not be re-defined; we just adapt it, using fixed-time relations.
\begin{rmk}
Let $\Phi$ be a multiflow over topological space $X$, and let $U\subset X$. Then for $t\ge0$
	\begin{align*}\konf(U;\Phi^t) =& \{K\subset X : K \hbox{ is closed and confining with respect to }\Phi^t,\\
						     & \hbox{ and there is some $n>0$ such that }\Phi^{nt}(U)\subset K\}.
	\end{align*}
\end{rmk}
For the latter - the omega limit set from a continuous time standpoint - we emphasize that $K$ being confining in a multiflows sense means: $\Phi^t(K)\subset K$ for {\it all} $t>0$.
\begin{defn}
Let $\Phi$ be a multiflow over topological space $X$, and let $U\subset X$. Then
	\begin{align*}
	\konf(U;\Phi) =& \{K\subset X : K \hbox{ is closed and confining with respect to }\Phi,\\
	                       &\hbox{ and there is some $t>0$ such that }\Phi^t(U)\subset K\}.
	\end{align*}
\end{defn}
When $U$ is understood, the shorthand of $\konf_{\Phi^t}$ or $\konf_t$ may be used to identify the $\konf$-set for a fixed time relation, and $\konf_\Phi$ or $\konf$ may be used for the $\konf$-set in a multiflow sense.

\begin{lem}
If $K\subset\konf(U;\Phi)$ for multiflow $\Phi$ over topological space $X$ and $U\subset X$, and $T>0$ is a number such that $\Phi^T(U)\subset K$, then $\Phi^t(U)\subset K$ for all $t>T$.
\end{lem}

\begin{proof}
This comes almost immediately from the definition of $K\in\konf(U;\Phi)$. Let $t=T+s>T$ be given, and observe:
	\[	\Phi^{t}(U)=\Phi^s(\Phi^T(U))\subset \Phi^s(K)\subset K,\]
	with the first inequality coming from $K\subset \konf(U;\Phi)$ and the second from the fact that $K$ is confining.
\end{proof}

\begin{thm}{(First proved by K.J. Meyer \ref{pc:KJMeyer})}
If $\fonf\subset\konf(U;\Phi)$ is a finite collection of $\konf$-sets, then $\bigcap\fonf\subset\konf(U;\Phi)$. That is, the intersection of a finite number of $\konf$-sets is itself a $\konf$-set.
\end{thm}

\begin{proof}
Let $\fonf = \{K_1,\ldots,K_n\}\subset \konf$. Three properties are necessary for $\bigcap\fonf=K\in\konf$:
\begin{itemize}
\item {\bf closed:} The intersection of closed sets is closed.
\item {\bf confining:} Each $K_i$ is confining, so
	$\Phi^t(K)\subset\Phi^t(K_i)\subset K_i$
	for all $i$, meaning $\Phi^t(K)\subset \bigcap\limits_{1\le i \le n}K_i = K$ for all $t>0$.
\item {\bf contains a forward image of $U$:}
For each $i$, $K_i\in\konf$, meaning there is some $t_i>0$ such that $\Phi^{t_i}(U)\subset K_i$. Set $T = \max\{t_i\}$, and there will be an $\alpha_i\ge0$ for each $i$ such that $T=\alpha_i+t_i$. Then,
	\[\Phi^T(U) = \Phi^{\alpha_i}( \Phi^{t_i}(U)) \subset \Phi^{\alpha_i}(K_i)\subset K_i\]
for each $K_i$. This demonstrates the number $T>0$ such that $\Phi^T(U)\subset K.$
\end{itemize}
Thus, $K=\bigcap\fonf\in\konf$.
\end{proof}

\begin{rmk}
It is the last property - that $K$ contains a forward image of $U$, $\Phi^t(U)$, with finite $t$ - which might prevent general intersections of $\konf$-sets from being in $\konf$.
\end{rmk}

\begin{thm}\label{thm:KonfNonEmpty}
Let $\Phi$ be a multiflow over a compact space $X$. Let $U\subset X$. Then $\konf(U;\Phi)\neq\es$; in fact, $X\in\konf(U;\Phi)$.
If in addition $B\overset{cl}{\subset}X$ is an eventually confining set with respect to $\Phi$, then 
	\[K=\bigcap\limits_{s\in[0,T]}\Phi^s(B)\in\konf(B;\Phi),\] 
where $T$ is a number such that $\Phi^t(B)\subset B$ for all $t\ge T$.
\end{thm}

\begin{proof}
Let $X$ and $\Phi$ be as stated in the theorem. Let $U\subset X$. Because $X$ is our ambient space, $\Phi^t(U)\subset X$ for all $t\in[0,\infty).$ Also, $X$ is closed in the relative topology, and $\Phi^t(X)\subset X$, because images under $\Phi$ are only defined on $X$.

More interestingly, let us consider  $T, B$, and $K$ as stated in the theorem. We note first that $K$ is closed because it is the intersection of closed sets. Next, let $s\in[0,T]$ be given. Then
\begin{align*}
	\Phi^{2T}(B)=&\Phi^{s+\alpha}(B),\hbox{ where }\alpha>T \\
			 =&\Phi^s(\Phi^\alpha)(B) \\
			 \subset & \Phi^s(B).
\end{align*}
So, $\Phi^{2T}(B)\subset \bigcap\limits_{s\in[0,T]}\Phi^s(B)=K$.

Finally, we show that $K$ is confining. Let $s\in[0,T]$, $t\in[0,\infty)$ and $x\in K$ be given. We consider two cases:

{\bf $t>s$:} Set $k=nT+s-t$ for some $n\in\ZZ_{>0}$ which gives $k\in[0,T]$. So $t+k=nT+s$. The image of $x$ satisfies:
	\begin{align*}
		\Phi^t(x)\subset&\Phi^t(\Phi^k(B)) \\
				=&\Phi^{t+k}(B)=\Phi^{nT+s}(B)\\
				=&\Phi^s(\Phi^{nT}(B))\subset\Phi^s(B).
	\end{align*}
	
{\bf $t\le s$:} Set $k=t-s\in[0,T]$. Then $x\in\bigcap\limits_{\tau\in[0,T]}\Phi^\tau(B)\subset\Phi^k(B)$, giving us
	\[ \Phi^t(x)\subset\Phi^t(\Phi^k(B)) = \Phi^{t+k}(B) = \Phi^s(B).\]
So, no matter the value of $t\in[0,\infty)$, $\Phi^t(K)\subset\bigcap\limits_{s\in[0,T]}\Phi^s(B)=K$, making $K$ confining. Thus, $K\in\konf(B;\Phi).$
\end{proof}

It is worth noting that an individual element of $\konf$, $K=\bigcap\limits_{s\in[0,T]}\Phi^s(B)$, can be the empty set, in which case $\es\in\konf$ (in fact, the forward image $\Phi^t(B)$ may be $\es$ for some $t>0$).

\begin{cor}\label{thm:ImgKonfInKonf}
Let $\Phi$ be a multiflow over a compact space $X$, and let $B\overset{cl}{\subset}X$ be eventually confining, with $T>0$ a number such that $\Phi^t(B)\subset B$ for all $t\ge T.$ Then all images $\Phi^s(K)$ where $K=\bigcap\limits_{\tau\in[0,T]}\Phi^\tau(B)$ ($s>0$) are also in $\konf(B;\Phi)$.
\end{cor}
\begin{proof}
Let $X, B, \Phi,$ and $T$ be as stated in the corollary. Recall from Theorem \ref{thm:KonfNonEmpty} that $K\in \konf(B;\Phi)$. Let $s\ge0$.
	\begin{itemize}
		\item $\Phi^s(K)$ is closed (the image of a closed set under a closed fixed time relation over a compact space).
		\item Recall from the proof of Theorem \ref{thm:KonfNonEmpty} that $\Phi^{2T}(B)\subset K$, which implies \[\Phi^{2T+s}(B)=\Phi^s(\Phi^{2T}(B))\subset \Phi^s(K).\]
		\item Finally, let $u>0$. Then $\Phi^u(\Phi^s(K))=\Phi^s(\Phi^u(K))\subset \Phi^s(K)$ because $K$ was confining for $\Phi$ (again, as shown in the proof for Theorem \ref{thm:KonfNonEmpty}). Thus, $\Phi^t(K)$ is confining for $\Phi$. 
	\end{itemize}
Thus, $\Phi^s(K)\in\konf(B;\Phi)$.
\end{proof}

\begin{lem}\label{thm:kSetInclusion}
Let $\Phi$ be a multiflow over a topological space $X$, and let $U\subset V\subset X$. Then $\konf(U;\Phi) \supset \konf(V;\Phi)$. \end{lem}

\begin{proof}
Let $X, U, V,$ and $\Phi$ be as described. Let $\konf_U=\konf(U;\Phi)$, $\konf_V=\konf(V;\Phi)$. Consider some $K\in\konf_V$. Then, it already satisfies the need to be closed and confining. Also, there is some $t>0$ for which $K\supset\Phi^t(V)\supset\Phi^t(U)$ (by Lemma \ref{lem:relninclusion}). So, $K\in\konf_U$, making $\konf_U\supset\konf_V$.
\end{proof}

\subsection{Omega Limit Sets for Multiflows}\label{sec:OmLimSetsDisp}

Now that we have laid out the $\konf$-sets, we may define omega limit sets for multiflows. With the relation $f$ over $X$, there are two objects to consider; with multiflows, that grows to four. There are two, which come naturally from fixed time relations: $\om(U;\Phi^t)$ and $\hat{\om}(U;\Phi^t)$. Recall that the latter is called the {\it strict} omega limit set. Then, there are the related definitions for $\Phi$, which use all values of $t>0$. All of the sets in the table are closed, as they are all intersections of closed sets.

\[\begin{array}{ c || l | l}
  & \hbox{{\bf fixed time}} & \hbox{{\bf multiflow}} \\
  & \hbox{{\bf relation }} \Phi^t                                       & \Phi \\
  \hline\hline &&\\
  \hbox{{\bf strict omega}} &\hat{\omega}(U;\Phi^t) \hspace{.65in}&\hat{\omega}(U;\Phi) \hspace{.5in}\\ \hbox{{\bf limit set}} & =\bigcap\limits_{m\ge0}\ol{\bigcup\limits_{n\ge m}\Phi^{n t}(U)} &= \bigcap\limits_{t\ge0}\ol{\Phi^{[t,\infty)}(U)} \\ &&\\
  \hline &&\\
  \hbox{{\bf omega limit}} & \om(U;\Phi^t)\hspace{.45in} & \om(U;\Phi)\hspace{.45in} \\ \hbox{{\bf set}} & =\bigcap \konf(U;\Phi^t) & = \bigcap \konf(U;\Phi) \\ &&
   \end{array}\label{table:omLimSets} \]
  \vspace{.3in}


Once again, the strict omega limit sets are not necessarily invariant.

\begin{ex}\label{ex:strOmNotInvt}
Consider the space $X$ is shown in Figure \ref{fig:strOmNotInvt}. We act upon it by a multiflow $\Phi$. The points in $U$ (the straight line pictured in Figure \ref{fig:strOmNotInvt}) move asymptotically toward the point $x$. The action on the point $x$ is such that one can remain at $x$ in future time, or one can move off the point $x$ onto the rest of the circle, which has period $T$. We consider $\hat{\om}(U;\Phi) = \{x\}$, but $\Phi^{[0,\infty)}(x)$ is the whole circle. That is, $\{x\}$ is not an invariant set (it is not confining). If we instead calculate $\om(U;\Phi) = \bigcap \konf(U;\Phi)$, then then we take the intersection of closed confining sets, which contain forward images of $U$. All such sets contain $x$, but they also hold the whole circle, because they must be confining. For any point $y\in U$, there is at least one $K\in\konf(U;\Phi)$, which does not contain $y$. So, $\om(U;\Phi)$ is the whole circle, which is invariant.

				\begin{figure}\centering
					\includegraphics[width=.5\linewidth]{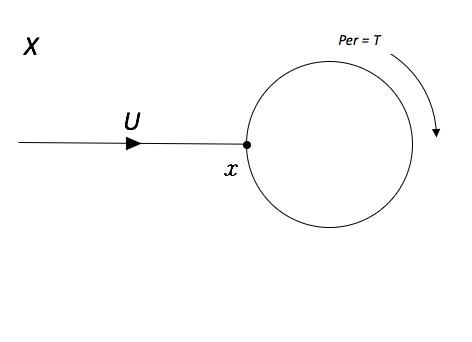}
					\caption{$\hat{\om}(U;\Phi)$ from Example \ref{ex:strOmNotInvt} is not invariant.}\label{fig:strOmNotInvt}
				\end{figure}

\end{ex}

Just as maps are useful in determining the end behavior of flows and semiflows, we will use the omega limit sets for relations to shed light on those defined directly through multiflows.

\begin{thm}\label{thm:OmHRiOmR}
If $\Phi$ is a multiflow over a compact Hausdorff space $X$ and $U\subset X$, then for any $t>0$,
	\[\hat{\om}(U;\Phi^t)\subset\om(U;\Phi^t).\]
\end{thm}

\begin{proof}
Apply Lemma \ref{thm:InclOmLimReln} with $f=\Phi^t$.
\end{proof}

\begin{thm}\label{thm:ECOmRiOmHR}
If $\Phi$ is a multiflow over a compact Hausdorff space $X$ and $U\subset X$ is {\em eventually confining} with respect to $\Phi$, then for any $t>0$,
	\[\hat{\om}(U;\Phi^t)=\om(U;\Phi^t).\]
\end{thm}

\begin{proof}
Apply Lemma \ref{ConfDispConfRel} and Theorem \ref{thm:EvConfRel} with $f=\Phi^t$.
\end{proof}

\begin{thm}\label{thm:OmHats} Let $\Phi$ be a multiflow over a space $X$. Let $U\subset X$. Fix any $t>0$, and
\[\hat{\om}(U;\Phi^t)\subset \hat{\om}(U;\Phi).\]
If, in addition, $U$ is eventually confining (in the multiflow sense), then we gain equality:
\[\hat{\om}(U;\Phi^t) = \hat{\om}(U;\Phi).\]
\end{thm}


\begin{proof}
Let $\Phi$ be a multiflow over $X$ and $U\subset X$. Fix some $\tau>0$ and consider $\hat{\om}(U;\Phi^\tau)$ and $\hat{\om}(U;\Phi)$. Let 
	\[\mathcal{S}=\left\{ \ol{\bigcup\limits_{n\ge m} \Phi^{n\tau}(U)}\right\}_{m\ge0},\]
	\[\mathcal{T}=\left\{ \ol{\bigcup\limits_{s\ge t} \Phi^{s}(U)}\right\}_{t\ge 0},\]
Consider any $S=\bigcup\limits_{n\ge m} \Phi^{n\tau}(U)=\Phi^{m\tau}(U)\cup\Phi^{(m+1)\tau}(U)\cup\ldots,$ so $\ol{S}\in\mathcal{S}$ There is a $T$ with $\ol{T}\in\mathcal{T}$ such that
	\[ T=\bigcup\limits_{s\ge m\tau}\Phi^s(U) \supset S,\]
and therefore $\ol{S}\subset\ol{T}$. By Lemma \ref{lem:intersections},
	\[ \hat{\om}(U;\Phi^\tau)=\bigcap\limits_{\ol{S}\in\mathcal{S}} \ol{S} \subset \bigcap\limits_{\ol{T}\in\mathcal{T}} \ol{T} = \hat{\om}(U;\Phi).\]
	
To prove the other inclusion, we add the assumption that $U$ is eventually confining. Let $t_0>0$ be such that $\Phi^t(U)\subset U$ for all $t\ge t_0$. Again, fix some $\tau>0$ and some $m>0$. Set $s\ge m\tau+t_0$. Then, $s=m\tau+t_0+\alpha$, and so
	\[\Phi^s(U)=\Phi^{m\tau+(t_0+\alpha)}(U)\subset \Phi^{m\tau}(U),\]
Thus, $\bigcup\limits_{s\ge m\tau+t_0}\Phi^s(U)\subset\Phi^{m\tau}(U)\subset\bigcup\limits_{k\ge m}\Phi^{k\tau}(U).$ The same inclusion holds for their closures. So, for each $\ol{S}\in\mathcal{S}$, there is some $\ol{T}\in\mathcal{T}$ with $T\subset S$. By Lemma \ref{lem:intersections},
	\[ \hat{\om}(U;\Phi^\tau)=\bigcap\limits_{\ol{S}\in\mathcal{S}} \ol{S} \supset \bigcap\limits_{\ol{T}\in\mathcal{T}} \ol{T} = \hat{\om}(U;\Phi),\]
giving us equality.
\end{proof}

These sets are not necessarily equal if $U$ is not confining with respect to $\Phi$, as will be shown in Example \ref{ex:NotEvConfOmNotEqual}.

\begin{thm}\label{thm:OmRiOmD} Let $\Phi$ be a multiflow over a space $X$. Let $U\subset X$. Then  \[\om(U;\Phi^t)\subset \om(U;\Phi)\] for any $t>0$.
\end{thm}

\begin{proof}
Let $\Phi$ and $U\subset X$ be as described, and let $t>0$ be fixed. Let
\begin{align*}\konf=&\{K\subset X: K\hbox{ is closed and confining for }\Phi,\\
		 &\hbox{ and there is some }\tau>0\hbox{ for which }\Phi^\tau(U)\subset K\},\\
		\konf_t=&\{K_t\subset X: K_t\hbox{ is closed and confining for }\Phi^t,\\ 
		&\hbox{ and there is some }n\in\ZZ_{>0}\hbox{ for which }\Phi^{nt}(U)\subset K_t\}.
\end{align*}
Consider some $K\in\konf$. We now show $K\in\konf_t$. Because $K$ is confining for $\Phi$ (in the multiflow sense), $\Phi^t(K)\subset K$, making $K$ confining for $\Phi^t$. Let $\tau>0$ be a number such that $\Phi^\tau(U)\subset K$. Then, find some integer $n>\frac{\tau}{t}$. Let $\alpha=nt-\tau>0$, and 
\begin{align*}
				\Phi^{nt}(U)=&\Phi^\alpha(\Phi^{\tau}(U))\\
				\subset& \Phi^\alpha(K)\\
				\subset&K.
\end{align*}
Thus $K\in K_t$, and 
	\[\konf\subset\konf_t \implies  \om(U;\Phi^t)=\bigcap\limits_{K_t\in \konf_t}K_t\subset\bigcap\limits_{K\in\konf}K=\om(U;\Phi)\]
	for any $t>0$.
\end{proof}

\begin{thm}
If $\Phi$ is a multiflow over a compact space $X$ and $U\subset X$, then $\hat{\om}(U;\Phi)\subset\om(U;\Phi)$.
\end{thm}

\begin{proof}
Let $X,U$, and $\Phi$ be as described. Recall that by Theorem \ref{thm:KonfNonEmpty}, $\konf(U;\Phi)\neq\es.$ Consider any $K\in\konf(U;\Phi)$. Then, $K$ is closed and confining, and there is some $T\ge0$ for which $\Phi^t(U)\subset K$ or all $t\ge T$. Thus,
	\[ \Phi^t(U)=\Phi^{t-T}(\Phi^T(U))\subset \Phi^{t-T}(K)\subset K \]
for all $t\ge T$, which implies $\bigcup\limits_{t\ge T} \Phi^t(U)\subset K$, but $K$ is closed, so $\ol{\bigcup\limits_{t\ge T} \Phi^t(U)}\subset K$. Finally, we see that
	\[\hat{\om}(U;\Phi)=\bigcap\limits_{s\ge0}\ol{\bigcup\limits_{t\ge s}\Phi^t(U)}\subset \ol{\bigcap\limits_{t\ge T}\Phi^t(U)}\subset K.\]
This is true for all $K\in \konf(U;\Phi)$, so $\hat{\om}(U;\Phi)\subset \bigcap \konf(U;\Phi)=\om(U;\Phi).$
\end{proof}


\begin{thm}\label{thm:ECOmDiOmHD}
If $\Phi$ is a multiflow over a compact space $X$ with $B\overset{cl}{\subset}X$ eventually confining with respect to $\Phi$, then $\om(B;\Phi)\subset\hat{\om}(B;\Phi).$
\end{thm}


\begin{proof}

Let $K_B=\bigcap\limits_{s\in[0,T]}\Phi^s(B)$ where $T$ is a number such that $\Phi^t(B)\subset B$ for all $t>T$. Then $\{\Phi^t(K_B)\}_{t\in[0,\infty)}\subset \konf(B;\Phi)$, as shown in Corallary \ref{thm:ImgKonfInKonf}. Note also that $K_B\subset B$. Thus,
	\[	\Phi^t(K_B)\subset \Phi^t(B)\subset \ol{\bigcup\limits_{s\ge t}\Phi^s(B)},\]
and so
	\[ 	\hat{\om}(B;\Phi)=\bigcap\limits_{t\ge0}\ol{\bigcup\limits_{s\ge t}\Phi^s(B)}\supseteq \bigcap\limits_{t\ge0}\Phi^t(K_B)\supseteq \bigcap\konf(B;\Phi)= \om(B;\Phi). \]

\end{proof}

To summarize the relationships, refer to the following corollaries.
\begin{cor}
 Let $\Phi$ be a multiflow over a compact Hausdorff space $X$ with $U\subset X$. Let $t>0$ be a fixed time. Then
\[ \begin{array}{ccc}  \hat{\om}(U;\Phi^t) & \subset                 &\hat{\om}(U;\Phi) \\ 
                                \rotatebox[origin=c]{270}{$\subset$} &   & \rotatebox[origin=c]{270}{$\subset$}        \\
                                \om(U;\Phi^t)          &   \subset                &\om(U;\Phi) .    \end{array} \]
\end{cor}
\begin{proof}
                               This is achieved by collecting Theorems \ref{thm:OmHRiOmR}, \ref{thm:OmHats}, and \ref{thm:OmRiOmD}. \end{proof}

\begin{cor}\label{thm:EvConfOmegaEqual}
If $\Phi$ is a multiflow over a compact Hausdorff space $X$ and $U\overset{cl}{\subset}X$ is eventually confining with respect to $\Phi$, then
	\[      \hat{\om}(U;\Phi^t) = \hat{\om}(U;\Phi)= \om(U;\Phi) = \om(U;\Phi^t) \]
for any fixed time $t>0$.
\end{cor}

\begin{proof}
Let everything be as described in the corallary. Then by Theorems \ref{thm:OmHRiOmR}, \ref{thm:ECOmRiOmHR}, \ref{thm:OmHats}, \ref{thm:OmRiOmD}, and \ref{thm:ECOmDiOmHD},  

\[ \begin{array}{ccc}  \hat{\om}(U;\Phi^t) & =                 &\hat{\om}(U;\Phi) \\ 
                                \rotatebox[origin=c]{270}{$=$} &   & \rotatebox[origin=c]{90}{$=$}        \\
                                \om(U;\Phi^t)          &                   &\om(U;\Phi) .    \end{array} \]
\end{proof}

It is useful to be able to switch the definition of omega limit set, depending on the application. So long as $U$ is eventually confining, they will all agree.

\begin{thm}\label{OmFTRelEq}
If $\Phi$ is a multiflow over a compact Hausdorff space $X$ and $U\subset X$ is {\em eventually confining} with respect to $\Phi$, then for any $s,t>0$,
	\[\om(U;\Phi^s)=\om(U;\Phi^t).\]
\end{thm}

\begin{proof}
This is a result of Corollary \ref{thm:EvConfOmegaEqual}: Let $U\subset X$ and $\Phi$ be as described. Then, no matter what $s,t>0$ are,
	\[	\om(U;\Phi^t)=\om(U;\Phi)=\om(U;\Phi^s).\]
\end{proof}

As the following example demonstrates, this is not necessarily the case when $U$ is not eventually confining.

\begin{ex}\label{ex:NotEvConfOmNotEqual} Let $X=\RR^2$, and consider the flow (and therefore a multiflow) we've already considered a few times. The vector field, which defines it is
		\[ \left[\begin{array}{c} \dot{x} \\ \dot{y} \end{array}\right] = \left[\begin{array}{cc} 0 & 1 \\ -1 & 0 \end{array}\right]\left[\begin{array}{c} x \\ y \end{array}\right],\] 
which is simply the rotation matrix. So, the solution $\Phi$ is the flow, which rotates about the origin with period $2\pi$. Let $U$ be an ellipse, say $4x^2+y^2 \le 16.$ It is not eventually confining. Set $s=\pi$ and $t = \pi/3$. Then $\hat{\om}(U;\Phi^s)=\om(U; \Phi^s) = U$ is the ellipse in Figure \ref{fig:omUPi}, while $\hat{\om}(U;\Phi^t)=\om(U;\Phi^t)$ is the flower shape in Figure \ref{fig:omUPiThirds}, and $\hat{\om}(U;\Phi)=\om(U;\Phi)$ is the circular region in Figure \ref{fig:omU}.

In this case,
	\[   \om(U;\Phi^s)\subsetneq\om(U;\Phi^t)\subsetneq\om(U;\Phi).\]
	
\begin{figure}[h]
    \centering
    \subfigure[$\om(U;\Phi^{\pi}).$]
    {
        \includegraphics[width=.25\linewidth]{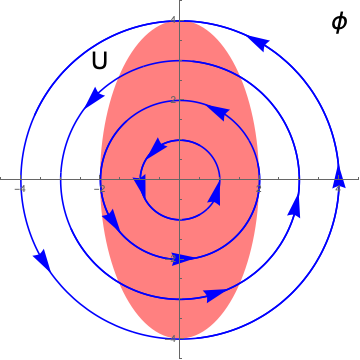}
        \label{fig:omUPi}
    }
    \subfigure[$\om(U;\Phi^{\pi/3}).$]
    {
        \includegraphics[width=.25\linewidth]{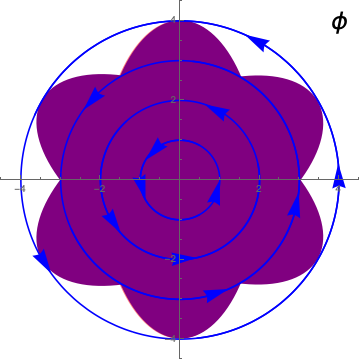}
        \label{fig:omUPiThirds}
    }
    \subfigure[$\om(U;\Phi).$]
    {
        \includegraphics[width=.25\linewidth]{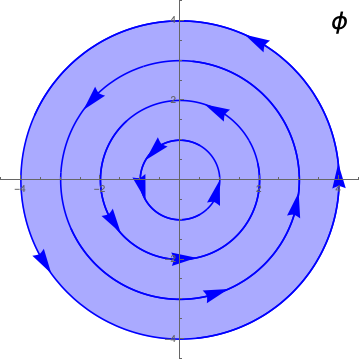}
        \label{fig:omU}
    }
    \label{fig:NotEvConfOmNotEqual}    
    \caption{Demonstrates the omega limit sets of $U$, as described in Example \ref{ex:NotEvConfOmNotEqual}.}
\end{figure}
	
\end{ex}

The omega limit set and the strict omega limit set for multiflows preserve inclusion. 

\begin{thm}\label{thm:DispOmPresIncn}
Let $\Phi$ be a multiflow over a space topological $X$. Let $U,V\subset X$. Then 
\begin{enumerate}	
	\item $U\subset V\implies \om(U;\Phi)\subset\om(V;\Phi)$, and
	\item	$U\subset V\implies \hat{\om}(U;\Phi)\subset\hat{\om}(V;\Phi).$
\end{enumerate}
\end{thm}

\begin{proof}
Assume $\Phi$ to be muliflow over a topological space $X$ with $U\subset V\subset X$. We prove the implications in order.
When we construct the omega limit set of $U$ or $V$, we take $\om(U;\Phi)=\bigcap\limits_{K\in\konf}K$ and $\om(V;\Phi)=\bigcap\limits_{L\in\lonf}L$, where
\begin{align*}
	\konf(U)=\{&K: K\hbox{ is closed and confining (in a multiflow sense), and there exists}\\						&\hbox{some }t_K\hbox{ for which }\Phi^{t_K}(U)\subset K\},\\
	\lonf(V)=\{&L: L\hbox{ is closed and confining (in a multiflow sense), and there exists}\\							&\hbox{some }t_L\hbox{ for which }\Phi^{t_L}(V)\subset L\}.
\end{align*}

Lemma \ref{thm:kSetInclusion} already guarantees that $\lonf\subset\konf$. By Lemma \ref{lem:intersections},
	\[\om(U;\Phi)=\bigcap\limits_{K\in\konf}K\subset\bigcap\limits_{L\in\lonf}L=\om(V;\Phi).\]
	
For the other inclusion, we have $\Phi^t(U)\subset\Phi^t(V)$ for each $t\ge0$ because relations preserve inclusion. Then,
	\begin{align*}
	\bigcup\limits_{t\ge s}\Phi^t(U)\subset&~ \bigcup\limits_{t\ge s}\Phi^t(V) \hbox{      for all }s\ge0\\
	\implies \ol{\bigcup\limits_{t\ge s}\Phi^t(U)}\subset& ~ \ol{\bigcup\limits_{t\ge s}\Phi^t(V)} \hbox{      for all }s\ge0\\
	\implies \hat{\om}(U;\Phi)=\bigcap\limits_{s\ge0}\ol{\bigcup\limits_{t\ge s}\Phi^t(U)}\subset& ~\bigcap\limits_{s\ge0}\ol{\bigcup\limits_{t\ge s}\Phi^t(V)}=\hat{\om}(V;\Phi).
	\end{align*}

\end{proof}

\begin{lem} \label{lem:ftRelnConf_Inc}
Let $\Phi$ be a multiflow over a compact Hausdorff space $X$ with confining set $B\subset X$. Consider the fixed times $s<t\in[0,\infty)$. The corresponding fixed time relations $\Phi^s$ and $\Phi^t$ satisfy the following:
\begin{enumerate}
\item $\Phi^t(B)\subset\Phi^s(B)$
\item $\Phi^t\big|_{B\times X}\subset\Phi^s\big|_{B\times X}$
\item $\om(B;\Phi^t)\subset\om(B;\Phi^s)$
\end{enumerate}
\end{lem}

\begin{proof}
First, we rewrite $t=s+\alpha$ where $\alpha>0$. Because $B$ is a confining set,
		\[\Phi^t(B)=\Phi^{s}(\Phi^\alpha(B))\subset\Phi^s(B).\]
This gives us (1). Now, if $(x,y)\in\Phi^s\big|_{B\times X}$, then $x\in B$ and $y\in \Phi^s(B)\subset\Phi^t(B)$. So, $(x,y)\in\Phi^t\big|_{B\times X}$; we have (2). 
Finally, we use Theorem 5.15 from \ref{bk:McGeheeAttractors}, and get (3).
\end{proof}

\subsection{Attractors and Attracting Neighborhoods for Multiflows}\label{sec:AttAndAttBlk}

It is worth noting that an attractor is a compact set, which is invariant, but which also pulls neighboring objects in toward it. A set qualifying as an attractor depends on having an attracting neighborhood. So, we define an attractor and an attracting neighborhood almost in the same breath. The main theorems of the section - Theorems \ref{thm:AttBlkGivesAtt} and \ref{thm:AttImpliesAttBlk} -
 link these two ideas.
 
\begin{defn}
Let $\Phi$ be a multiflow over a space $X$. Then an {\bf attractor} $A$ for $\Phi$ is a compact invariant set, which has a compact neighborhood $U$ such that $\om(U;\Phi)=A$.
\end{defn}

\begin{defn}
Let $\Phi$ be a multiflow over a space $X$. Then an {\bf attracting neighborhood $B$} associated to $A$ if it is a compact neighborhood of $A$ satisfying
\begin{enumerate}
	\item $\om(B;\Phi)=A$
	\item $B$ is eventually strictly confining.
\end{enumerate}
\end{defn}

\begin{thm}\label{thm:AttMaxInvt}
Let $\Phi$ be a multiflow over a compact Hausdorff space $X$, and $A=\om(U;\Phi)$, where $A\subset U\subset X$. Then $A$ is maximally invariant in $U$.
\end{thm}

\begin{lem}\label{thm:InvtEqOwnOm}
If $\Phi$ is a multiflow over a compact Hausdorff space $X$, and $S$ is a compact invariant set, then \[\om(S;\Phi) = S.\]
\end{lem}

\begin{proof}
Because $S$ is invariant, it is also eventually confining, so by Corollary \ref{thm:EvConfOmegaEqual}
	\[\om(S;\Phi)=\hat{\om}(S;\Phi) = \bigcap\limits_{s\ge0}\ol{\bigcup\limits_{t\ge s}\Phi^t(S)} = \bigcap\limits_{s\ge0}\ol{S} = \bigcap\limits_{s\ge0}S = S.				\]
\end{proof}

\begin{lem}\label{thm:InvtUnionIntersection}
Say $\Phi$ is a multiflow over a space $X$ and $\frak{A}$ is a set of subsets $A\subset X$, which are invariant (with respect to $\Phi$), then
	\begin{itemize}
		\item $\bigcup \frak{A}$ is invariant, and
		\item $\bigcap \frak{A}$ is confining.
	\end{itemize}
\end{lem}

\begin{proof}
Let $\Phi$, $X$ and $\frak{A}$ be as described. Then for each $A\in\frak{A}$, $A$ is invariant for each fixed time relation $\Phi^t$, $t\ge0$. Thus, for each $t\ge0$, 
	\[   \Phi^t\left(\bigcup\frak{A}\right)=\bigcup\frak{A}\hbox{ and }\Phi^t\left(\bigcap\frak{A}\right)\subset\bigcap\frak{A}\] 
by Theorem \ref{thm:RelnInvcUandInt}.

\end{proof}

\begin{lem}\label{thm:AttractorInvt}
Let $\Phi$ be a multiflow over a compact Hausdorff space $X$. Then if $A$ is an attractor for $\Phi$, then it is an invariant set, which has a neighborhood $U$ such that $\om(U;\Phi^t)=A$ for all fixed time relations $\Phi^t$, $t>0$.
\end{lem}

\begin{proof}
Let's first assume that $A$ is an attractor. That ensures the existence of a neighborhood $U'$ of $A$ such that $\om(U';\Phi)=A$. By Theorem \ref{thm:AttImpliesAttBlk}, there is an eventually confining neighborhood $U$ of $A$ (with $U\subset U'$). By Corollary \ref{thm:EvConfOmegaEqual}, $A=\om(U;\Phi)=\om(U;\Phi^t)$ for all $t>0$. Since we are able to define $A$ by the relations definition of the omega limit set, we have that $A$ is invariant for $\Phi^t$ for all $t>0$, meaning $A$ is invariant in the multiflow sense.

\end{proof}

\begin{proof}{\it(of Theorem \ref{thm:AttMaxInvt})}

First, Lemma \ref{thm:AttractorInvt} guarantees that $A$ is invariant. Next, let us assume $A$ is not the maximal invariant set in $U$. Then, there is some other invariant set $A'\subset U$, which contains elements not in $A$. Lemma \ref{thm:InvtUnionIntersection} states that $A\cup A'$ is also invariant, so we may as well assume $A\subset A'$. Then, we have
	\begin{align*} A\subset A' \subset U \implies&  \om(A) \subset \om(A')\subset \om(U) \hbox{ (Theorem \ref{thm:DispOmPresIncn})}  \\
								\implies& A\subset A' \subset A \hbox{ (Lemma \ref{thm:InvtEqOwnOm})}.
	\end{align*}
Thus, $A=A'$, but we assumed $A'$ had elements, which were not in $A$, giving us our contradiction. Therefore, $A$ must be maximally invariant.
\end{proof}

\begin{thm}\label{thm:AttBlkGivesAtt}
Let $\Phi$ be a multiflow over a compact Hausdorff space $X$, and let $B$ be a closed attracting neighborhood. Then, $B$ contains a unique attractor, $A=\om(B;\Phi)$.
\end{thm}

\begin{proof}
Let $\Phi$, $X$, and $B$ be as described. Set $A=\om(B;\Phi)$. Because $B$ is eventually confining $A=\om(B;\Phi)=\om(B;\Phi^t)$ for all $t>0$ (by Corollary \ref{thm:EvConfOmegaEqual}). Recall that the latter is the omega limit set for the fixed-time relation $\Phi^t$, which is invariant under $\Phi^t$. Thus $\Phi^t(A)=A$ for all $t>0$, making $A$ invariant.

For the other half of the definition, consider $B$, which we will show is a neighborhood of $A$. Because $B$ is eventually strictly confining, there is some $T>0$ for which $\Phi^t(B)\subset \Int(B)$ for all $t>T$. In particular, $\Phi^T(B)\subset\Int(B)$. 
By Theorem \ref{thm:KonfNonEmpty}, the set $K_T = \bigcap\limits_{t\in[0,T]}\Phi^t(B)\in\konf(B;\Phi)$. Thus,
	\[ A = \bigcap\konf(B;\Phi) \subset K_T \subset \Phi^T(B)\subset \Int(B).\]

\end{proof}

There was another version of this proof, constructed before Corollary \ref{thm:EvConfOmegaEqual} (on which this proof relies) was proven. It is provided in Appendix \ref{sec:AppendixA}. The proof is no longer necessary, but there are techniques included there, which may prove useful for studying multiflows in the future. Also, using that proof means one can take for granted that $A=\om(B)$ is invariant. We do not get that for free here. It is true, but this fact requires a bit of preamble.

\begin{thm}\label{thm:AttImpliesAttBlk}
Let $\Phi$ be a multiflow over a compact Hausdorff space $X$. Let $A$ be an attractor for $\Phi$ and let $V$ be a neighborhood of $A$. Then, there exists a closed attracting neighborhood $B\subset V$ associated with $A$.
\end{thm}

\begin{lem}\label{thm:ClosedNbdofAtt}
If $\Phi$ is a multiflow over a compact Hausdorff space $X$ and $A$ is an attractor for $\Phi$, then there is a closed neighborhood $N$ of $A$ such that $\om(N)=A$.
\end{lem}

\begin{proof}
Because $A$ is an attractor, we are guaranteed some neighborhood $N'$ of $A$ with $\om(N')=A$. There is guaranteed to be a closed set $N$ such that
	\[ A\subset\Int(N)\subset N\subset N'.\]
By Lemma \ref{thm:InvtEqOwnOm} and Theorem \ref{thm:DispOmPresIncn},
	\[ A=\om(A;\Phi)\subset\om(N;\Phi)\subset\om(N';\Phi)=A.\]
\end{proof}

\begin{lem}{Theorem 2.8 from \ref{bk:McGeheeAttractors}}\label{thm:finIntnNbhd}
If $\konf$ is a set of closed subsets of a compact topological space and if $U$ is a neighborhood of $\bigcap\konf$, then there exists a finite subset $\frak{F}$ of $\konf$ such that $\frak{F}\subset U$.
\end{lem}

\begin{proof}(of Theorem \ref{thm:AttImpliesAttBlk})
Let $\Phi,$ $X$, and $A$ be as described. Let $V$, a neighborhood of the attractor $A$, be given. Since we could always find a closed neighborhood $V'$ of $A$ with $V'\subset V$, we may as well assume $V$ is closed. Because $A$ is an attractor, there is some neighborhood $N$ of $A$ such that $\om(N;\Phi)=A$. By Lemma \ref{thm:ClosedNbdofAtt}, we may as well assume $N$ is closed, too. Let $B=N\cap V$. Clearly then $B$ is closed and $B\subset V$. We will show in addition that
\begin{enumerate}
	\item $\om(B;\Phi)=A$
	\item $B$ is eventually strictly confining.
\end{enumerate}

For the first item, we note that 
	\[A\subset (N\cap V) = B\subset N \implies  A = \om(A;\Phi)\subset \om(B;\Phi)\subset \om(N;\Phi)=A.\]
	
For the second, we consider $\Int(N\cap V)=\Int(B)$, which is also a neighborhood of $A=\bigcap\konf(B;\Phi)$. By Lemma \ref{thm:finIntnNbhd}, there exists some finite subset $\{K_i\}_{1\le i \le n}=\frak{F}\subset \konf(B;\Phi)$ such that $\bigcap\frak{F}\subset \Int(B)$. For each $K_i\in\frak{F}$ there is a $t_i$ such that $\Phi^{t_i}(B)\subset K_i$, and all the $K_i$ are confining. Let $T=\max\{t_i\}_{1\le i\le n}$, and $\Phi^t(B)\subset K_i$ for all $t\ge T$ and $1\le i \le n$, meaning that
	\[ \Phi^t(B)\subset \bigcap\frak{F} \subset \Int(B)\]
for all $t\ge T$. So, $B$ is eventually strictly confining.
\end{proof}

Looking at smaller time intervals can be useful in determining attracting behavior over all time, as the following theorem demonstrates.

\begin{thm}\label{Thm:AttBlkForwardTime}
Let $\Phi\cl\subset [0,\infty)\times X\times X$ be a multiflow. If $B\subset X$ is an attractor block for all the fixed time relations $\{\Phi^s\}_{s\in(0,T]}$, for some $T>0$. Then $B$ is an attractor block for all $\{\Phi^t\}_{t\in(0,\infty)}$.
\end{thm}

\begin{lem}\label{thm:AttBlk_DiscForwardTime}
Let $\Phi\cl\subset [0,\infty)\times X\times X$ be a multiflow. If $B\subset X$ is an attractor block for the relation $\Phi^T$ for some fixed time $T>0$, then $B$ is an attractor block for all relations in the set $\{\Phi^{n T}\}_{n\in\ZZ_{+}}$.
\end{lem}

\begin{proof}
This is simply a proof by induction: $\Phi^T(\ol{B})\subset \Int(B).$ Then for any $n\in\ZZ_+$, assume $\Phi^{(n-1)T}(\ol{B})\subset\Int(B),$ and 
	\begin{align*}
		\Phi^{nT}(\ol{B}) =& ~\Phi^T\left( \Phi^{(n-1)T}(\ol{B})\right) \\
					\subset& ~\Phi^T\left( \Int(B) \right)\hbox{ (by the previous step)} \\
					\subset& ~\Phi^T(\ol{B}) \subset \Int(B).
	\end{align*}
Thus, $B$ is an attractor block for all fixed time relations in the set $\{\Phi^{nT}\}_{n\in\ZZ_+}$.
\end{proof}

\begin{proof}[Proof of Theorem \ref{Thm:AttBlkForwardTime}]
The proof is similar to that of the lemma. Let $t>0$ be any fixed forward time, and let it be rewritten as $t=nT+s$, where $n\in \ZZ_{\ge0}$ and $s\in[0,T)$. Then,
	\begin{align*}
		\Phi^t(\ol{B})=& ~ \Phi^{nT+s}(\ol{B}) \\
				=& ~ \Phi^s\left(\Phi^{nT}(\ol{B})\right) \\
				\subset& ~ \Phi^s(\Int(B))\hbox{         (by Lemma \ref{thm:AttBlk_DiscForwardTime})}\\
				\subset& \Phi^s(\ol{B}) \\
				\subset& ~ \Int(B),
	\end{align*}
making $B$ an attractor block for the relation $\Phi^t$.
\end{proof}

\begin{lem}
Let $\Phi$ be a multiflow over a topological space $X$, and let $B\subset X$ be a set such that $\ol{B}$ is eventually strictly confining (there exists a $T>0$ such that $\Phi^t(\ol{B})\subset\Int(B)$ for $t\ge T$). Then,

	\[ B\hbox{ and }\Int(B)\hbox{ are eventually confining.}\]
If, in addition, $X$ is a compact Hausdorff space, then
	\[\om(\ol{B};\Phi)=\om(B;\Phi)=\om(\Int(B);\Phi).\]
\end{lem}

Similar things to (1) can be said of sets which are eventually strictly rejecting, with the inclusions going in the opposite directions. 

\begin{proof}
For the first item, let $T>0$ be a number such that $\Phi^t(\ol{B})\subset\Int(B)$ for all $t\ge T$.
Then for any given $t\ge T$,
	\[	\Phi^t(\Int(B))\subset\Phi^t(B)\subset\Phi^t(\ol{B})\subset\Int(B)\subset B.		\]
For the second item, let $f=\Phi^T$. By Corollary \ref{thm:EvConfOmegaEqual} and Lemma \ref{thm:RelnsOmCandInt},\begin{align*}
	\om(B;\Phi)=\om(\Int(B);f)=&\om(B;f)=\om(B;\Phi) \\
					=&\om(\ol{B};f)=\om(\ol{B};\Phi).
\end{align*}
\end{proof}

Finally, it may be useful to classify the behavior over more general sets - sets, where all we know initially is that they contain their own omega limit sets.

\begin{thm}
If $\Phi$ is a multiflow over a compact space $X$, and $U\subset X$. Then,
\begin{enumerate}
	\item $\om(U;\Phi)\subset \Int(U)\implies U$ is eventually strictly confining.
	\item $\om(U;\Phi)\subset U\implies U$ is eventually confining.
\end{enumerate}
\end{thm}


\begin{proof}
First, let us assume that $\Phi$ is a multiflow over $X$. Then we start by proving (1), so we assume $U\subset X$ and $\om(U;\Phi)\subset \Int(U)$. Then, let us also assume that $U$ is not eventually strictly confining. So, for any $s>0$, there will always be a $\tau\ge s$ such that $\Phi^\tau(U)$ has some points outside $\Int(U)$, i.e., $\Phi^\tau(U)\cap \ol{U^c}\neq\es$. Thus, 
	\[ \es\neq \left(\Phi^\tau(U)\cap \ol{U^c} \right) \subset \left( \bigcup\limits_{t\ge s}\Phi^t(U)\cap \ol{U^c}  \right) \subset  \left( \ol{\bigcup\limits_{t\ge s}\Phi^t(U)}\cap \ol{U^c}  \right) \]
for all $s>0$. Therefore,
	\[ \es\neq \left(  \bigcap\limits_{s\ge 0}\ol{\bigcup\limits_{t\ge s} \Phi^t(U)} \cap \ol{U^c} \right) = \left(\hat{\om}(U;\Phi)\cap\ol{U^c}\right) \subset \left(\om(U;\Phi)\cap\ol{U^c}\right).\]
This implies $\om(U;\Phi)\not\subset \Int(U)$, which is contrary to our other assumptions, so $U$ must be eventually strictly confining.

We prove (2) by exactly the same method. Assuming $U$ is not eventually confining means that for all $s>0$ there is some $\tau\ge s$ such that $\Phi^\tau(U)$ has some points outside $U$, meaning $\Phi^\tau(U)\cap U^c\neq\es$. The rest follows just as in the proof of (1), except that we use $U^c$, rather than its closure.

\end{proof}


\newpage
\section{On the Semicontinuity of Attractors for Relations}\label{sec:CloseRelns}

Attractors for relations are semicontinuous in the following way: let $X$ be a compact metric space (so we have some notion of distance), $f$ a closed relation over $X$, and $A_f\subset X$ an attractor for $f$. For any neighborhood $\hat{N}_\alpha\subset X$ of $A_f$ ($\alpha>0$), there is an attractor block $B\subset \hat{N}_\alpha$ associated to $A_f$ (that is, $B$ is an attractor block and $\om(B;f)=A_f\subset\Int(B)$). As we will show in this section, we can find a neighborhood $\ol{N}_\ep(f)\subset X\times X$ ($\ep>0$) such that for any closed relation $g\subset \ol{N}_\ep$, $B$ is also an attractor block for $g$. Therefore, we can find an attractor for $g$: $$A_g\subset B\subset\hat{N}_\alpha.$$ Note that $\alpha$ can be as small as desired, and we will still be able to find a close enough family of relations, which also have attractors close enough to $A_f$.

Let $X$ be a compact metric space. Unless otherwise specified, whenever $x,~x'\in X$, $d(x,x')$ means distance in $X$ with the metric already provided. Likewise, the default for whenever $(x,y),~(x',y')\in X\times X$ is the Euclidean distance in $X\times X$: $$d((x,y),(x',y'))=\sqrt{(d(x,x'))^2+(d(y,y'))^2}.$$ The distance between sets is then calculated in the usual way. That is, for $S,S'\subset X$ where $X$ is some metric space,
	\[ d(S,S')=\inf\{d(x,x'):~ x\in S,~x'\in S'\}, \]
	where $\inf$ is the infemum. 
	
\begin{rmk} When deciding whether or not two relations are close {\bf to each other}, we use the Hausdorff metric: given two relations $f,g\subset X\times X$ over a metric space $X$, $d_H(f,g)$ is defined as the smallest $\ep\ge0$ such that
	\[ f\subset \ol{N_\ep}(g), \hbox{ and } g\subset \ol{N_\ep}(f), \]
where $N_\ep(S)$ is the closed $\ep$-neighborhood of a set. Our notion of ``semicontinuity" of relations is half of what is required for closeness in the Hausdorff metric; that is, we only require one containment. So, $g$ is ``{\bf close to $f$}" if $g\subset N_\ep(f)$ for some small $\ep>0$. 
\end{rmk}

\begin{rmk} Also recall that for a relation $f\overset{cl}\subset X\times X$, $B$ is an attractor block for $f$ if and only if $f\cap(\ol{B}\times\ol{B^C})=\es$. \ref{bk:McGeheeAttractors}

\end{rmk}

Let $X$ be a compact metric space, $f\overset{cl}{\subset}X\times X$ a relation over $X$. Define a function $F_f$ over the space of closed relations over $X$, $\ol{\mathcal{R}}_X$ by
	\begin{align*}      F_f:\ol{\mathcal{R}}_X\rightarrow& [0,\infty)\\
								g	\mapsto& \inf\{\ep\ge0: g\subset \ol{N}_\ep(f)\} \end{align*}

\begin{thm}\label{thm:CloseRelnsAttBlk}
Let $f,~g\overset{cl}\subset X\times X$ be two (closed) relations. Let $B$ be an attractor block for $f$. If the closed relation $g$ is in a small enough neighborhood (in $X\times X$) of $f$, then $B$ is also an attractor block for $g$.\end{thm}

\begin{rmk}[Initial thoughts]
The initial intuition may be to consider the distance between $\ol{B}$ and the closure of its image, $\ol{f(B)}$. That is, we wonder if it is enough to set $\delta=d(\ol{B},\ol{f(B)})>0$, and if $d_H(f,g)=\ep<\delta$, then $B$ would also be an attractor block for $g$. The following counterexample is provided, showing that we really do need to consider the distances in $X\times X$ and not just in the image set.
\end{rmk}

\begin{ex}\label{Ex:CE1}
Counterexample: Let $X=[0,3].$ Consider $$f=\left([0.8,2+\alpha]\times\{1.5\}\right) \cup \left(\{2+\alpha\}\times[1.5,3] \right).$$ Then $B=[1,2]$ is an attractor block: $$\ol{B}\times\ol{B^c} = [1,2]\times\left([0,1]\cup[2,3]\right),$$ which has no overlap with $f$. Furthermore, $d(\ol{f(B)},\ol{B^c})=0.5=\delta$, while $d( f, \ol{B}\times\ol{B^c}) = \min\{0.5, \alpha\}$. See Figure \ref{Fig:CE1} for an illustration in which $\alpha = 0.1.$

Note that $\alpha$ can be arbitrarily small without changing the nature of $B$. So, let $\alpha = 0.1<0.5.$ Let us take $\ep$ to be between them: $\alpha<\ep<\delta$. In this case, $\ep=0.15$ works. Set
	\[ g = \ol{N_{\ep}}(f)=\ol{N_{0.15}}(f).\]
Then, $g$ is within $0.5$ of $f$, but $g\cap (\ol{B}\times\ol{B^c}) \neq \es$, meaning $B$ fails to be an attractor block for $g$.

\begin{figure}\centering[h]\label{Fig:CE1}
  \includegraphics[width=3.5in]{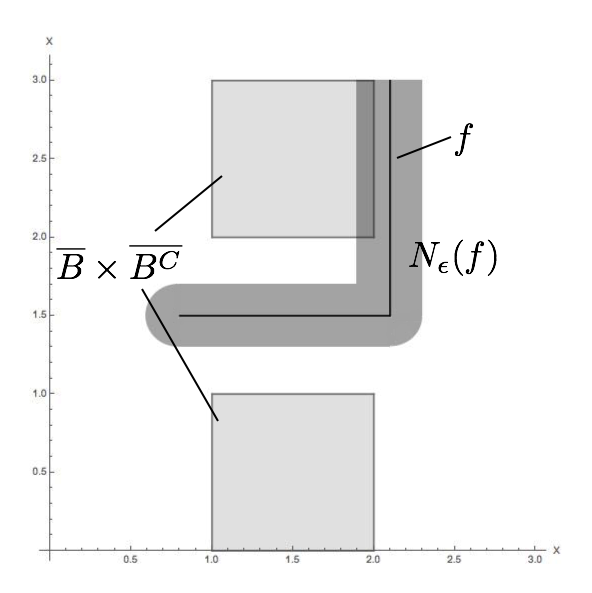}\label{fig:MapEx1f}
  \caption{Illustration of $f$, $N_\ep(f)$, and $\ol{B}\times\ol{B^c}$ from Example \ref{Ex:CE1}.}
\end{figure}

Scaling $\ep$ based on $d( f, \ol{B}\times\ol{B^c})$ won't work, either. The choice of $\alpha$ could have been $10^{-100}$ and we could still choose $\ep$ with $\alpha<\ep<\delta$, without changing the nature of the counterexample.

\end{ex}

\begin{lem}\label{lem:RelnBubbleAttBlk}
Let $f\overset{cl}\subset X\times X$ be a closed relation with $B$ an attractor block for $f$. Let $\delta = d( f, \ol{B}\times\ol{B^c} )>0$. Then $\ol{N_\ep}(f)\cap (\ol{B}\times\ol{B^c})=\es$ for any $\ep<\delta$.
\end{lem}

\begin{proof}
This is a straight-forward result. The distance from $f$ to any point in $\ol{B}\times \ol{B^c}$ is at least $\delta$, so any point in $\ol{N_\ep}(f)$ has not strayed far enough from $f$ to land in $\ol{B}\times\ol{B^c}$.
\end{proof}

\begin{proof}[Proof of Theorem \ref{thm:CloseRelnsAttBlk}] Let $d_H(f,g)=\ep$. We claim that if $\ep<\delta=d(f,\ol{B}\times\ol{B^c})$, then $B$ is also an attractor block for $g$. By definition of the Hausdorff metric, we know 
	\[g\subset N_\ep(f)\subset \ol{N_\ep}(f) \subset N_\delta(f).\]
	Thus, by Lemma \ref{lem:RelnBubbleAttBlk}, $g\cap(\ol{B}\times\ol{B^c})\subset N_\delta(f) \cap (\ol{B}\times\ol{B^c}) = \es,$ making $B$ an attractor block for $g$ as well.

\end{proof}	

Technically, one does not require the full strength of closeness in the Hausdorff metric; one side of the inclusion will suffice.

\begin{thm}\label{thm:eBallRelnsAttBlk}
Let $f,g\overset{cl}\subset X\times X$ be closed relations with $B$ an attractor block for $f$. Let $\delta = d( f, \ol{B}\times\ol{B^c} )>0$. If $g\subset \ol{N_\ep}(f)$, then $B$ is an attractor block for $g$.
\end{thm}

The proof for Theorem \ref{thm:CloseRelnsAttBlk} also works for Theorem \ref{thm:eBallRelnsAttBlk}.

This result is key in defining the continuation properties pioneered by C. Conley. These theorems imply that nearby relations will have similar end behaviors. Therefore, a relation can be altered, but in a small enough way that the topological properties do not change. 

Let us say, for instance, that we have a sequence of maps $\{f_n\}$ over the compact metric space $X$, whose graphs limit to the relation $f$ (whether or not $f$ is a graph of a map). Then, say $f$ has an attractor $A$. Choose a neighborhood $V$ of $A$, and we are guaranteed the existence an attractor block $B\subset V$, associated to $A$. We can then find $\delta(B)$, as described in Lemma \ref{lem:RelnBubbleAttBlk}, and fix $0<\ep<\delta(B)$. For any $f_i$ in the sequence, whose graph satisfies $\text{graph }\subset N_\ep(f)$, $B$ is also an attractor block, implying that those maps $f_i$ also have attractors in $B$.

\newpage
\section*{References}\label{references_chapter}
\begin{enumerate}[label={[\arabic*]},itemsep=0mm]
\item P. Collet, J.P. Eckmann, {\it Iterated Maps on the Interval as Dynamical Systems}, Birkh\"{a}user, Boston, (1980).\label{bk:ColletEckmann}
\item C. Conley, {\it Isolated Invariant Sets and the Morse Index}, Reg. Conf. in Math. {\bf 38} CBMS (1978).\label{bk:Conley}
\item M. Di Bernardo, C.J. Budd, A.R. Champneys \& P. Kowalczyk,
{\it Piecewise-smooth Dynamical Systems Theory and Applications}, Springer (2008).\label{bk:BBCK}
\item A. F. Filippov, {\it Differential Equations with Discontinuous Righthand Sides}, Dept. Math., Moscow State University, U.S.S.R., Kluwer Acad. Pub., Boston, MA (1988).\label{bk:Filippov}
\item J. Guckenheimer, Piecewise-Smooth Dynamical Systems, SIAM Review {\bf 50} (2008), pp. 606-609. \label{bk:Guckenheimer}
\item J. Leifeld, {\it Smooth and Nonsmooth Bifurcations in Welander's Convection Model}, thesis, University of Minnesota (2016). \label{bk:Leifeld}
\item R. McGehee, {\it Attractors for Closed Relations on Compact Hausdorff Spaces}, IN U. Math. Journal {\bf 41} {\it 4} (1992).\label{bk:McGeheeAttractors}
\item R. McGehee, personal communication (2015-2016).\label{pc:McGehee}
\item R. McGehee, E. Sander {\it A New Proof of the Stable Manifold Thm}, ZAMP {\bf 74} {\it 4} (1996), pp. 497-513. \label{bk:McGeheeSander}
\item R. McGehee, C. Thieme, personal communication (2017-2018). \label{pc:Cameron}
\item R. McGehee, T. Wiandt {\it Conley Decomposition for Closed Relations}, preprint (2005).\label{bk:McGeheeWiandt}
\item K. J. Meyer, personal communication (2016-2018). \label{pc:KJMeyer}
\item K. Mischaikow, {\it The Conley Theory: A Brief Introduction}, Center for Dyn. Sys. and Nonlinear Studies, Georgia Inst. Tech., Atlanta, GA, Banach Center Publications, Vol **, Inst. of Math., Polish Acad. of Sci. Warszawa {\bf 199*} (1991).\label{bk:Mischaikow}

\end{enumerate}

\newpage
\appendix
\section{An Alternate Proof.}\label{sec:AppendixA}


Note that with the proof provided in Section \ref{sec:AttAndAttBlk} necessitates that one shows elsewhere that the resulting attractor $A$ is invariant. With this proof, though it is less elegant, that fact comes for free.

\begin{thm}\label{thm:efConfSet!Att}
Let $\Phi$ be a multiflow over a compact Hausdorff space $X$, and let $B$ be an attracting neighborhood for $\Phi$ (so $B$ is eventually strictly confining).
\end{thm}

\begin{proof}
Let $\Phi$, $X$, and $B$ be as described in the theorem; in particular $B$ is eventually strictly confining. Let $t_0>0$ be a number such that $\Phi^t(B)\subset \Int(B)$ for all $t\ge t_0$.
We propose that given any $t>0$, $\om(B,\Phi^t)=A$ is the attractor associated with $B$ under $\Phi$. 

Certainly, $A$ is the attractor associated with $B$ under the fixed time relation $\Phi^t$. 
The question is: does $A$ serve as the attractor associated with $B$ for all relations $\Phi^s$ ($s>0$)?
We consider the cases in which $\frac{s}{t}$ is rational and irrational.

\noindent{\bf Case 1: $\frac{s}{t}\in\QQ$}
Find $p,~q\in\ZZ_+$ such that $\frac{s}{t}=\frac{p}{q}$. Then, $qs=pt$, and so $\Phi^{qs}=\Phi^{pt}$ are the same relation. By Theorem 5.11 in \ref{bk:McGeheeAttractors},
	\[ \om(B;\Phi^s)=\om(B;\Phi^{qs})=\om(B;\Phi^{pt})=\om(B;\Phi^t).\]


\noindent{\bf Case 2: $\frac{s}{t}\notin\QQ$}

We employ Dirichlet's approximation theorem, which states that there exist arbitrarily large $p,~q\in\ZZ_+$ for which $\left| \frac{s}{t}-\frac{p}{q}\right|<\frac{1}{q^2}$. Then $\left| sq-pt\right|<\frac{t}{q}$. Since $t$ is fixed, we can find a $p$ and a $q$ so that $\frac{t}{q}$ is arbitrarily small, say less than some $\alpha\in(0,1)$. At the very least, we want $\alpha<1$. Because $p$ and $q$ are arbitrarily large, let us also require that $p=p_1+p_2$, where $p_2 t > t_0$ and $p_1>0$. 

Now we can say without loss of generality $0< sq - pt =\beta< \alpha$. Then $sq = pt+\beta = p_1t+p_2t+\beta$, and so $\Phi^{sq}=\Phi^{p_1t}(\Phi^{p_2t+\beta})$ are the same relation.

Because $p_2 t >t_0$,
\begin{align} 
	\Phi^{qs}\big|_{B\times X}= \Phi^{p_1t + p_2t+\beta}\big|_{B\times X} =& \Phi^{p_1 t}(\Phi^{p_2 t+\beta})\big|_{B\times X}\subset  \Phi^{p_1t}\big|_{B\times X}.
\end{align}
By Theorems 5.15, 5.11, and 5.10 of \ref{bk:McGeheeAttractors},
\begin{align*}
\om(B;\Phi^s) =&~ \om(B;\Phi^{qs}) \hspace{.95in}\hbox{by Thm 5.11}\\
					=&~ \om\left(B;\Phi^{qs}\big|_{B\times X}\right) \hspace{.5in}\hbox{by Thm 5.15}\\
					\subset&~ \om\left(B;\Phi^{p_1t}\big|_{B\times X}\right) \hspace{.5in}\hbox{by Thm 5.10 and (2)}\\
					=&~ \om(B;\Phi^{p_1t})\\
					=&~ \om(B;\Phi^t)=A.
\end{align*}
For the other inclusion, we keep the values of $p, p_1, p_2,$ and $q$, and we find $N\in\ZZ_{>0}$ such that $Np_2t - p_1 t - \beta>t_0$. Then, 
	\[ p_2 t = qs - p_1t -\beta \implies (N+1)p_2 t = qs  + Np_2t - p_1t - \beta ,\]
	and
\begin{align}
	\Phi^{(N+1)p_2 t}\big|_{B\times X} =& \Phi^{qs}(\Phi^{Np_2t-p_1t-\beta})\big|_{B\times X} \subset \Phi^{qs}\big|_{B\times X}.
\end{align}

	Following the same logic, we get
\begin{align*}
\om(B;\Phi^t) =&~ \om(B;\Phi^{(N+1)p_2t}) \\
					=&~ \om\left(B;\Phi^{(N+1)p_2t}\big|_{B\times X}\right) \\
					\subset&~ \om\left(B;\Phi^{qs}\big|_{B\times X}\right) \hspace{.5in}\hbox{by Thm 5.10 and (3)}\\ 
					=&~ \om(B;\Phi^{qs})\\
					=&~ \om(B;\Phi^s).
\end{align*}

Again using Theorem 5.9, $A$ is invariant under the closed relations $\Phi^t$ for any $t\ge0$. Also, $B$ is a neighborhood of $A$ by Theorem 7.2 from \ref{bk:McGeheeAttractors}, for which $\om(B;\Phi^t)=A$ for all $t>0$. Thus $A$ is the unique attractor for any fixed time relations $\Phi^t$, $t>0$.\end{proof}

\end{document}